\documentclass[10pt,twoside,dvips]{article}
\usepackage[a5paper,a5paper,layoutwidth=148.5mm,layoutheight=7.77817in,inner=50pt,outer=39pt,top=40pt,bottom=28pt,
footskip=18pt
]{geometry}
\usepackage{amsmath,amssymb,amsthm}
\usepackage{mathtools}
\mathtoolsset{mathic}

\usepackage[utf8]{inputenc}
\usepackage[T1]{fontenc}
\usepackage{lmodern}
\usepackage{microtype}
\usepackage[german,french,british]{babel}

\usepackage[dvipsnames]{xcolor}

\newcommand\urlprefix{}

\PassOptionsToPackage{hyphens}{url}

\usepackage[bookmarksopen=false,breaklinks=true,plainpages=false,hyperindex=true,pdfstartview=FitH,colorlinks=true,pdfpagelabels=true,colorlinks=true,linkcolor=blue,citecolor=red,urlcolor=red,hypertexnames=false,
backref=page,pagebackref=true,
]{hyperref}

\renewcommand*{\backref}[1]{}
\renewcommand*{\backrefalt}[4]{\ifcase #1\or(p.~#2)\else(pp.~#2)\fi}
\renewcommand*{\backreftwosep}[1]{ and }
\renewcommand*{\backreflastsep}[1]{, and }

\usepackage{doi}\usepackage{breakurl}
\newtheorem{theorem}{Theorem}[section]
\newtheorem{proposition}[theorem]{Proposition}
\newtheorem{lemma}[theorem]{Lemma}
\newtheorem{corollary}[theorem]{Corollary}

\theoremstyle{remark}
\newtheorem{remark}[theorem]{Remark}
\newtheorem{example}[theorem]{Example}

\theoremstyle{definition}
\newtheorem{definition}[theorem]{Definition}
\newtheorem*{gencontext}{General context}
\newtheorem{context}[theorem]{Context}

\newtheorem{definota}[theorem]{Definition and notation}

\newtheorem{divalgorithm}[theorem]{Division algorithm}
\newtheorem{spolyalgorithm}[theorem]{S-polynomial algorithm}
\newtheorem{buchbalgorithm}[theorem]{Buchberger's algorithm}
\newtheorem{syzalgorithm}[theorem]{Syzygy algorithm for terms}
\newtheorem{Syzalgorithm}[theorem]{Schreyer's syzygy algorithm}
\newtheorem*{problem}{Problem}

\numberwithin{equation}{section}

\newcommand{\so}[1]{\left\{#1\right\}}
\newcommand{\sotq}[2]{\{\,#1\mathrel;\allowbreak#2\,\}}
\newcommand \sumtq[2]{\sum_{#2}{#1}}
\newcommand \som {\sum\nolimits}

\DeclareMathOperator{\LT}{LT}
\DeclareMathOperator{\LM}{LM}
\DeclareMathOperator{\LC}{LC}
\DeclareMathOperator{\LP}{LPos}
\DeclareMathOperator{\Rad}{Rad}
\DeclareMathOperator{\Syz}{Syz}
\DeclareMathOperator{\mdeg}{mdeg}
\DeclareMathOperator{\Ann}{Ann}
\DeclareMathOperator{\ann}{ann}
\DeclareMathOperator{\lcm}{lcm}
\DeclareMathOperator{\Ker}{Ker}

\newcommand{\rS}{\mathrm{S}}
\newcommand \NN{\mathbb{N}}
\newcommand \ZZ{\mathbb{Z}}

\newcommand \uX  {\underline{X}}

\newcommand \lst[1]{(#1)}

\newcommand \gen[1]{\langle{#1}\rangle}
\newcommand \gentq[2]{\gen{\,#1\mathrel;#2\,}}

\newcommand \F {\mathbf{F}}
\renewcommand \H {\mathbf{H}}
\newcommand \K {\mathbf{K}}
\newcommand \R {\mathbf{R}}
\newcommand \V {\mathbf{R}}
\newcommand \Hm {\H_{m}}

\newcommand \noi {\noindent}
\newcommand \sms {\smallskip}
\newcommand \sni {\sms\noi}

\newcommand \bs {\bigskip}
\newcommand \bni {\bs\noi}

\newcommand{\intervalles}{[ \negthinspace [0 , N-1 ] \negthinspace ]}
\newcommand{\intervalless}{[ \negthinspace [1 , N-1 ] \negthinspace ]}

\providecommand{\divides}{\mathbin|}

\usepackage{enumitem}
\setlist[itemize, 1]{wide}
\setlist[enumerate, 1]{wide}
\setlist[enumerate, 2]{wide,leftmargin = \parindent}%

\newcommand{\+}{\mkern.5\thinmuskip}
\usepackage{listings}
\usepackage{authblk}

\newcommand{\X}[1][1]{\ifcase#1\or X\or Y\fi}

\begin{document}

\tolerance 1414
\hbadness 1414
\emergencystretch 1.5em
\hfuzz 0.3pt
\widowpenalty=10000
\vfuzz \hfuzz
\raggedbottom

\title{
  The syzygy theorem for Bézout rings}





\author[*]{Maroua Gamanda}
\author[**]{Henri Lombardi}
\author[**,1]{Stefan Neuwirth}
\author[*,2]{Ihsen Yengui}
\affil[*]{Département de mathématiques,  Faculté des sciences de Sfax, Université de Sfax, 3000 Sfax, Tunisia, \url{marwa.gmenda@hotmail.com}, \url{ihsen.yengui@fss.rnu.tn}.}
\affil[**]{Laboratoire de mathématiques de Besançon, Université Bourgogne Franche-Comté, 25030 Besançon Cedex, France, \url{henri.lombardi@univ-fcomte.fr}, \url{stefan.neuwirth@univ-fcomte.fr}.}
\begingroup
  \renewcommand\thefootnote{}
  \footnotetext{The statement of Theorems \ref{hilbvaluation1}, \ref{hilbvaluation0divis}, \ref{hilbBezout1} has been amended with respect to the published version, as well as the free resolution at the end of Example \ref{ex0002}. The changes have been typeset in {\color{OliveGreen}green}.}
\endgroup
\refstepcounter{footnote}\footnotetext{Supported by the French ``Investissements d’avenir'' program, project ISITE-BFC (contract ANR-15-IDEX-03).}
\refstepcounter{footnote}\footnotetext{Supported by  the John Templeton Foundation (ID~60842).}
\date{}
\maketitle

\date{}


\begin{abstract}
We provide constructive versions of Hilbert's syzygy theorem for
 $\ZZ$ and $\ZZ /N\ZZ $ following \foreignlanguage{german}{Schreyer}'s method. Moreover, we extend these results to arbitrary coherent strict Bézout rings with a divisibility test for the case of finitely generated modules whose module of leading terms is finitely generated.
\end{abstract}

  \bni MSC 2010:
13D02, 
13P10, 
13C10, 
13P20, 
14Q20

  \sni Keywords:
  Syzygy theorem, free resolution, monomial order, \foreignlanguage{german}{Schreyer}'s monomial order, \foreignlanguage{german}{Schreyer}'s syzygy algorithm, dynamical Gröbner basis, valuation ring, Gröbner ring, strict Bézout ring.

\tableofcontents


\section*{Introduction}\label{Introduction}
\addcontentsline{toc}{section}{Introduction}

This paper is written in the framework of Bishop style constructive mathematics (see~\cite{Bi67,BB85,ACMC,MRR}).
It can be seen as a sequel to the papers~\cite{HY,Ye2}. The main goal is to obtain constructive versions of Hilbert's syzygy theorem for Bézout domains of  Krull dimension~$\leq 1$  with a divisibility test and for coherent zero-dimensional Bézout rings with a divisibility test (e.g.\ for $\ZZ $ and $\ZZ /N\ZZ $, see~\cite{ACMC,MoY,Y4,Y5}) following \foreignlanguage{german}{Schreyer}'s method. These two cases are instances of Gröbner rings. Moreover, we extend these results to arbitrary coherent strict Bézout rings with a divisibility test for the case of finitely generated modules whose module of leading terms is finitely generated.


\section{Gröbner bases for modules over a discrete ring}\label{s1}

\begin{gencontext}
  In this article, $\R$ is a commutative ring with unit, $X_{1},\dots,X_{n}$ are $n$~indeterminates ($n\geq 1$), $\R[\uX]=\R[X_{1},\dots,X_{n}]$, $\Hm\simeq\mathbb{A}^{m}(\R[\uX])$ is a free $\R[\uX]$-module with basis $(e_{1},\dots,e_{m})$ ($m\geq 1$), and $>$~is a monomial order on~$\Hm$ (see Definition~\ref{mononmial-module}).
\end{gencontext}

We start with recalling the following constructive definitions.

\begin{definition}\label{defdef}\leavevmode
\begin{itemize}

\item $\R$ is \emph{discrete} if it is equipped with a zero test: equality is decidable.

\item $\R$  is \emph{zero-dimensional} and we write $\dim\R\leq0$ if
\[
  \forall a\in \R\ \exists k\in\NN\ \exists x\in\R\ \ a^{k}(ax-1)=0.
\]
\item $\R$ has Krull dimension $\leq 1$ and we write $\dim\R\leq1$  if
\[
  \forall a,b\in \R\ \exists k,\ell\in\NN\ \exists x,y\in\R\ \ b^{\ell}(a^{k}(ax-1)+by)=0.
\]
\item Let $U$ be an $\R$-module.
The \emph{syzygy module} of a $p$-tuple $\lst{v_{1},\dots,v_{p}}\in U^{p}$ is
\[
  \Syz(v_{1},\dots,v_{p})\coloneqq\sotq{ \lst{b_{1},\dots,b_{p}}\in \R^{n} }{ b_{1}v_{1}+\cdots+b_{p}v_{p}=0}\text.
\]
The syzygy module of a single element~$v$ is the \emph{annihilator} $\Ann(v)$ of~$v$.

\item An $\R$-module $U$ is \emph{coherent} if the syzygy module of
  every $p$-tuple of elements of~$U$ is finitely generated,\footnote{In contradistinction to Bourbaki and to the Stacks project, we do not require $U$~to be finitely generated.} i.e.\ if
  there is an algorithm providing a finite system of generators for
  the syzygies, and an algorithm that represents each syzygy as a
  linear combination of the generators. $\R$ is \emph{coherent} if it is coherent as an $\R$-module. It is well known that a module is coherent iff on the one hand any intersection of two finitely generated submodules is finitely generated, and on the other hand the annihilator of every element is a finitely generated ideal.

\item $\R$ is \emph{local} if, for every element $x\in \R$, either~$x$ or~$1+x$ is invertible.

\item $\R$ is \emph{equipped with
a  divisibility test} if,  given $a,b \in
\R$, one can answer the question $a\mathrel{\in?}\gen{b}$ and, in the case of
a positive answer, one can explicitly provide~$c \in \R$ such that
$a=bc$.

\item $\R$ is \emph{strongly discrete} if it is  equipped with a membership test for
finitely generated ideals, i.e.\ if, given  $a,b_{1},\dots,b_{p} \in \R$,
one can answer the question $a\mathrel{\in?}\gen{b_{1},\dots,b_{p}}$ and, in
the case of a positive answer, one can explicitly provide $c_{1},\dots,c_{p} \in
\R$ such that $a=c_{1}b_{1}+\cdots +c_{p}b_{p}$.

\item $\R$ is a \emph{valuation ring}\footnote{Here we follow Kaplansky's definition: $\R$ may have nonzero zerodivisors. In~\cite{Y5} it is required that a valuation ring be strongly discrete. We prefer to add this hypothesis when the argument requires it, so as to discriminate the algorithms that rely on the divisibility test from those that do not.} if  every two elements are comparable w.r.t.\ division, i.e.\ if, given $a,b \in \R$, either $a \divides b$ or $b \divides a$. A valuation ring is a local ring; it is coherent iff the annihilator of any element is principal. A valuation domain is coherent. A valuation ring is strongly discrete iff it is equipped with a divisibility test.

\item $\R$ is a \emph{Bézout ring} if every
finitely generated ideal is  principal, i.e.\ of the form $\gen a=\R a$ with $a \in \R$. A Bézout ring is strongly discrete iff it is equipped with a divisibility test; it is coherent iff the annihilator of any element is principal. To be a valuation ring is to be a Bézout local ring (see~\cite[Lemma IV-7.1]{ACMC}).

\item A Bézout ring $\R$ is \emph{strict} if for all $b_1, b_2 \in\R$ we can find
$d, b'_{1} , b'_{2} , c_1, c_2 \in\R$ such that $b_1 = d b'_{1}$, $b_2 = d b'_{2}$, and $c_1 b'_{1} + c_2 b'_{2} = 1$. Valuation rings and Bézout domains are strict Bézout rings; a quotient or a localisation of a strict Bézout ring is again a strict Bézout ring (see \cite[Exercise~IV-7 pp.~220--221, solution pp.~227--228]{ACMC}). A zero-dimensional Bézout ring is strict (because it is a ``Smith ring'', see \cite[Exercice~XVI-9 p.~355, solution p.~526]{DLQ2014} and \cite[Exercise~IV-8 pp.~221-222, solution p.~228]{ACMC}).

\end{itemize}
\end{definition}

\begin{remark} \label{remstrongdis}
In some cases, e.g.\ euclidean domains or polynomial rings over a discrete field, a strongly discrete ring is equipped with a \emph{division algorithm} which, for arbitrary $a\in\R$ and $(b_{1},\dots,b_{p}) \in \R^p$, provides an expression $a=c'_{1}b_{1}+\cdots +c'_{p}b_{p}+e$ with \emph{quotients} $c'_{1},\dots,c'_{p}$
and a \emph{remainder} $e$, where $e=0$ iff $a\in\gen{b_{1},\dots,b_{p}}$.
When a strongly discrete ring is not equipped with a division algorithm, we shall consider that the division is trivial
if $a\notin\gen{b_{1},\dots,b_{p}}$: the quotients vanish and $e=a$.
In the case of Bézout rings, dividing $a$ by  $\lst{b_{1},\dots,b_{p}}$ amounts to dividing $a$ by the gcd~$d$ of $\lst{b_{1},\dots,b_{p}}$, since $a=c'd+e$ can be read as $a=(c'c_{1})b_{1}+\cdots+(c'c_{p})b_{p}+e$,  where $d=c_{1}b_{1}+\dots+c_{p}b_{p}$.
\end{remark}

\begin{definition}[Monomial orders on finite-rank free {$\R[\uX]$}-modules, see~\cite{AL,coxlittleoshea05}]\label{mononmial-module}
\leavevmode

\begin{enumerate}
\item\label{mononmial-module-1}
  A \emph{monomial} in $\Hm$ is a vector of type $\uX^{\alpha} e_{\ell}$ ($1 \leq \ell \leq m$), where $\uX^{\alpha}=X_{1}^{\alpha_{1}}\cdots X_{n}^{\alpha_{n}}$ is a monomial in $\R[\uX]$;
  the index~$\ell$ is the \emph{position} of the monomial. The set of monomials in $\Hm$ is denoted by $\mathbb{M}_{n}^{m}$, with $\mathbb{M}_{n}^{1} \cong \mathbb{M}_{n}$ (the set of monomials in $\R[\uX]$). E.g., $X_{1}X_{2}^{3}e_{2}$ is a monomial in $\Hm$, but $2X_{1}e_{3}$, $(X_{1}+X_{2}^{3})e_{2}$ and $X_{1}e_{2}+X_{2}^{3}e_{3}$ are not.

  If $M=\uX^{\alpha} e_{\ell}$ and $N=\uX^{\beta} e_{k}$, we say that $M$ \emph{divides} $N$ if $\ell=k$ and $\uX^{\alpha}$ divides $\uX^{\beta}$. E.g., $X_{1}e_{1}$ divides $X_{1}X_{2}e_{1}$, but does not divide $X_{1}X_{2}e_{2}$. Note that in the case that $M$ divides~$N$, there exists a monomial $\uX^{\gamma}$ in $\R[\uX]$ such that $N=\uX^{\gamma} M$: in this case we define
\(
  N/M\coloneqq \uX^{\gamma}
\);
e.g., $(X_{1}X_{2}e_{1})/(X_{1}e_{1})=X_{2}$.

A \emph{term} in $\Hm$ is a vector of type $c M$, where $c \in \R \setminus \so{ 0 }$ and $M \in \mathbb{M}_{n}^{m}$. We say that a term $c M$ \emph{divides} a term $c' M'$, with $c,c' \in \R \setminus \so{0}$ and $M,M' \in \mathbb{M}_{n}^{m}$, if $c$ divides $c'$ and $M$ divides~$M'$.

\item A \emph{monomial order} on $\Hm$ is a
relation $>$ on $\mathbb{M}_{n}^{m}$ such that
\begin{enumerate}
\item $>$ is  a total order on $\mathbb{M}_{n}^{m}$,
\item $ \uX^{\alpha} M > M$ for all $M \in
\mathbb{M}_{n}^{m}$ and $ \uX^{\alpha} \in  \mathbb{M}_{n} \setminus \so{1}$,
\item $M>  N \implies \uX^{\alpha} M > \uX^{\alpha} N $ for all $M,N \in
\mathbb{M}_{n}^{m}$ and $ \uX^{\alpha} \in  \mathbb{M}_{n}$.
\end{enumerate}
Note that, when specialised to the case $m=1$, this definition coincides with the definition of a monomial order on $\R[\uX]$.

When $\R$ is discrete, any \emph{nonzero} vector $h \in \Hm$ can be written as a sum of terms
\[
  h=c_{t}M_{t}+c_{t-1}M_{t-1}+\dots+c_{1}M_{1},
\]
with $c_{\ell} \in \R \setminus \so{0}$, $M_{\ell} \in \mathbb{M}_{n}^{m}$, and $M_{t} > M_{t-1} > \cdots > M_{1}$. We define the \emph{leading coefficient}, \emph{leading monomial}, and \emph{leading term} of $h$ as in the ring case: $\LC (h)=c_{t}$, $\LM (h)=M_{t}$, $\LT (h)=c_{t}M_{t}$. Letting $M_{t}=\uX^{\alpha} e_{\ell}$ with $\uX^{\alpha}\in \mathbb{M}_{n}^{m}$ and $1\leq \ell \leq m$, we say that $\alpha$ is the \emph{multidegree of $h$} and write $\mdeg(h)=\alpha$, and that
the index~$\ell$ is the \emph{leading position} of~$h$ and write $\LP(h)=\ell$.

We stipulate that $\LT (0)=0$ and $\mdeg(0)=-\infty$, but we do not define $\LP(0)$.

\item A monomial order on
$\R[\uX]$ gives rise to the following canonical
monomial order on~$\Hm$: for monomials $M=\uX^{\alpha} e_{\ell}$ and $N=\uX^{\beta} e_{k} \in
\mathbb{M}_{n}^{m}$, let us define that
\[
  M>N \quad \text{if}\quad\left|\,
\begin{aligned}
\text{either }&\uX^{\alpha} > \uX^{\beta} \\
\text{or both }&\uX^{\alpha} = \uX^{\beta}\text{ and $\ell<k$.}
  \end{aligned}
\right.
\]
This monomial order is called \emph{term over position} (TOP) because it
gives more importance to the monomial order on $\R[\uX]$
than to the vector position. E.g., when  $X_{2}>X_{1}$,
we have
\[
  X_{2}e_{1}> X_{2}e_{2} > X_{1}e_{1} > X_{1}e_{2}\text.
\]
\end{enumerate}
\end{definition}


\begin{definition}[Gröbner bases and \foreignlanguage{german}{Schreyer}'s monomial order]\label{defiGrob}
  Let $\R$ be a discrete ring. Consider $G=\lst{g_{1},\dots,g_{p}}$, $g_{j}\in {\Hm\setminus \so0}$, and the finitely generated submodule $U= \gen{g_{1},\dots,g_{p}}=\R[\uX] g_{1}+\cdots+\R[\uX] g_{p}$  of $\Hm$.
  \begin{enumerate}
  \item The \emph{module of leading terms} of~$U$ is $\LT (U)\coloneqq\gentq{\LT(u)}{u \in U}$.

\item $G$ is a \emph{Gröbner basis} for $U$ if $\LT(U)= \gen{\LT (G)}\coloneqq \gen{\LT (g_{1}),\dots,\LT(g_{p})}$.

\item Let $(\epsilon_{1},\dots,\epsilon_{p})$ be the canonical basis of~$\R[\uX]^{p}$. \emph{\foreignlanguage{german}{Schreyer}'s monomial order induced by~\(>\) and \(\lst{g_{1},\dots,g_{p}}\)} on~$\R[\uX]^{p}$ is the order denoted by $>_{g_{1},\dots,g_{p}}$, or again by~$>$, defined as follows:
\[
  \uX^{\alpha}\epsilon_{\ell}>\uX^{\beta}\epsilon_{k}\quad\text{if}\quad\left|\,
    \begin{aligned}
      \text{either }&\LM (\uX^{\alpha} g_{\ell}) > \LM (\uX^{\beta} g_{k})\\
      \text{or both }&\LM (\uX^{\alpha} g_{\ell}) = \LM (\uX^{\beta} g_{k}) \text{ and $\ell<k$.}
    \end{aligned}\right.
\]
\end{enumerate}
\end{definition}

\foreignlanguage{german}{Schreyer}'s monomial order is defined on~$\R[\uX]^{p}$ in the same way as when~$\R$ is a discrete field (see \cite[p.~66]{EH}).

\section{The algorithms}\label{sec:generic-algorithms}

\subsection{The context}\label{sec:context}

Let us now present the algorithms to be discussed in this article in a form that adapts as well to the case where $\R$ is a coherent valuation ring with a divisibility test as to the case where $\R$ is a coherent strict Bézout ring with a divisibility test (note that the former case is the local case of the latter).
This is achieved by appeals to ``\lstinline|find $\dots$ such that $\dots$|'' commands that will adapt to the corresponding framework.
I.e., the following context is needed for the algorithms, except that coherence and strictness is not used in the division algorithm and that the divisibility test is not used for the computation of S-polynomials.

\begin{context} \label{context}
The algorithms take place in a coherent strict Bézout ring $\R$ with a divisibility test. In the local case,~$\R$ is a coherent valuation ring with a divisibility test. 
\end{context}

\subsection{The division algorithm}\label{sec:division-algorithm}

This algorithm takes place in Context~\ref{context} for $\R$;  note however that coherence and strictness are not used here.
Like the classical division algorithm for $\F[\uX]^{m}$ with $\F$ a discrete field (see~\cite[Algorithm~211]{Y5}), this algorithm has the following goal.
 \[
 \begin{tabular}{rl}
   \lstinline|Input|&$h\in\Hm$, $h_{1},\dots,h_{p} \in \Hm\setminus \so 0$.\\
   \lstinline|Output|&$q_{1},\dots,q_{p} \in \R[\uX]$ and $r \in \Hm$ such that\\
         &$\left\{\begin{aligned}
             &h=q_{1}h_{1}+\cdots+q_{p}h_{p}+r\text,\\
             &\LM (h)\geq \LM (q_{j})\LM(h_{j})\text{ whenever $q_{j}\ne0$,}\\
             &T\notin \gen{\LT(h_1),\dots,\LT (h_p)} \text{ for each term~$T$ of $r$.}
       \end{aligned}\right.$
 \end{tabular}
\]

\begin{definota} \label{notaremainder}
  The vector $r$ is called \emph{a remainder of \(h\) on division by \(H= \lst{h_{1},\dots,h_{p}}\)} and is denoted by $r= \overline{h}^{H}$.\end{definota}

This notation would gain in precision if it included the dependence of the remainder on the algorithm mentioned in Remark~\ref{remstrongdis}.


\begin{divalgorithm}\label{gendivalg}\leavevmode
\begin{lstlisting}
local variables $j:\so{1,\dots,p}$, $D:\text{subset of }\so{1,\dots,p}$,
                $c,c_{j},d,e:\R$, $h':\Hm$;
$q_{1}\gets0$; ... ; $q_{p}\gets0$; $r\gets0$; $h'\gets h$;
while $h'\ne0$ do
  $D\gets\sotq{j}{\LM (h_{j}) \divides \LM (h')}$;
  find $d,c_{j}\ (j\in D)$ such that
    $d = \gcd({\LC(h_{j})})_{j\in D} = \sumtq{c_{j}\LC(h_{j})}{j\in D}$; (*@\label{gendivalg:5}@*)
  find $c,e$ such that
    $\LC(h') = c d + e$ (*@{\rm (with $e=0$ iff $d$ divides $\LC(h'),$ see Remark~\ref{remstrongdis})}@*); (*@\label{gendivalg:6}@*)
  for $j$ in $D$ do
    $q_{j} \gets q_{j} + c c_{j} (\LM (h') / \LM(h_{j}))$
  od;
  $r \gets r + e \LM(h')$;
  $h'\gets {h' - \sumtq {c c_{j} (\LM(h') / \LM(h_{j})) h_{j}}{j\in D} - e \LM(h')}$
od
\end{lstlisting}
By convention, if $D$~is empty, then $d=0$. At each step of the algorithm, the equality $h=q_{1}h_{1}+\cdots+q_{p}h_{p} + h' +r$ holds while $\mdeg(h')$ decreases.\end{divalgorithm}

Note that in the case of a valuation ring, the gcd $d$ is an $\LC(h_{j_{0}})$ dividing all the $\LC(h_{j})$, and the Bézout identity may be given by setting $c_{j_{0}}=1$ and $c_{j}=0$ for~$j\neq j_{0}$: see Algorithm~\ref{def222}.

\subsection{The S-polynomial algorithm}\label{sec:s-polyn-algor}

This algorithm takes also place in Context~\ref{context} for $\R$. Note however that the divisibility test is not used here; only the zero test is used.
This algorithm is a key tool for constructing a Gröbner
basis and has been introduced by Buchberger~\cite{Buchbe0} for the case
where the base ring is a discrete field. It has the following goal.
 \[\makebox[\textwidth]{%
 \begin{tabular}{rl}
   \lstinline|Input|&$f,g \in \Hm\setminus\{0\}$.\\
   \lstinline|Output|&the S-polynomial given by $b\uX^{\beta}$ and $a\uX^{\alpha}$ as $\rS(f,g)=b\uX^{\beta}f - a\uX^{\alpha}g$:\\
&\enskip if $f=g$, then $b\uX^{\beta}=b$ is a generator of~$\Ann(\LC(f))$ and $a\uX^{\alpha}=0$;\\
                    &\enskip otherwise, if $\LM(f)=\uX^{\mu} e_{i}$ and  $\LM (g)=\uX^{\nu} e_{i}$, then $\rS(f,g)=$\\
                    &\enskip $b\uX^{(\nu-\mu)^{+}}f -a\uX^{(\mu-\nu)^{+}}g$ with $b\LC(f)=a\LC(g)$, $\gcd(a,b)=1$;\\
                    &\enskip otherwise, $\rS(f,g)=0$.
 \end{tabular}}
\]
Here $\alpha^{+}=(\max(\alpha_{1},0),\dots,\max(\alpha_{n},0))$ is the \emph{positive part} of~$\alpha\in\ZZ^{n}$.

\begin{spolyalgorithm}\label{genSalg}
  \leavevmode
\begin{lstlisting}
local variables $a,b:\R$, $\mu,\nu:\NN^{n}$;
if $f = g$ then(*@\label{genSalg:1}@*)
  find $b$ such that $\Ann(\LC(f))=\gen{b}$; (*@\label{genSalg:2}@*)
  $\rS(f,f) \gets b{}f$ (*@\label{genSalg:3}@*)
else(*@\label{genSalg:4}@*)
  if $\LP(f) \ne \LP(g)$ then
    $\rS(f,g) \gets 0$
  else
    $\mu \gets \mdeg(f)$; $\nu \gets \mdeg(g)$;
    find $a,b$ such that(*@\label{genSalg:11}@*)
      $\gcd(a,b)=1,$
      $a\gcd(\LC(f),\LC(g))=\LC(f),$
      $b\gcd(\LC(f),\LC(g))=\LC(g)$;(*@\label{genSalg:12}@*)
    $\rS(f,g) \gets b\uX^{(\nu-\mu)^{+}}f - a\uX^{(\mu-\nu)^{+}}g$
  fi
fi(*@\label{genSalg:15}@*)
\end{lstlisting}
\end{spolyalgorithm}
Note the following important properties of~$\rS(f,g)$:
\begin{itemize}
\item If $\LM (f)=\uX^{\mu} e_{i}$ and  $\LM (g)=\uX^{\nu} e_{i}$, then either $\rS(f,g)=0$ or
$\LM(\rS(f,g))< \uX^{\sup(\mu,\nu)}e_{i}$; if $\LP(f) \ne \LP(g)$, then $\rS(f,g)=0$;
\item  $\rS(\uX^{\delta} f,\uX^{\delta} g)=\uX^{\delta}\rS(f,g)$  for all $\delta \in \mathbb{N}^{n}$.
\end{itemize}

$\rS(f,f)$ is called the \emph{auto-S-polynomial} of $f$. It is
designed to produce cancellation of the leading term of $f$ by
multiplying $f$ with a generator of the annihilator of~$\LC (f)$. If the leading
coefficient of $f$ is regular, then $\rS(f,f)=0$ as in
the discrete field case. In case $\R$ is a domain, this algorithm is
not supposed to compute auto-S-polynomials and we can remove
lines~\ref{genSalg:1}--\ref{genSalg:4} and~\ref{genSalg:15}: if
nevertheless executed with $f=g$, it yields $\rS(f,f)=0$.

The S-polynomial $\rS(f,g)$ is designed to produce cancellation of the leading terms of~$f$ and~$g$. It is worth pointing out that $\rS(f,g)$ is not uniquely determined (up to a unit) when $\R$ has nonzero zerodivisors. Also $\rS(g,f)$ is generally not equal (up to a unit) to $\rS(f,g)$ (in the discrete field case, this ambiguity is taken care of by making the S-polynomial monic). These issues are repaired through the consideration of the auto-S-polynomials $\rS(f,f)$ and $\rS(g,g)$.

Note that in the case of a valuation ring, the computation of the coefficients $a,b$ is particularly easy: see Algorithm~\ref{def222Spaire}.

\subsection{Buchberger's algorithm}\label{sec:buchb-algor}

This algorithm takes place in Context~\ref{context} for $\R$. Here coherence, strictness, and the divisibility test are used.
Concerning the termination of the algorithm, see Section~\ref{case:Bezout-rings}.

This algorithm has the following goal.
 \[
 \begin{tabular}{rl}
   \lstinline|Input|&$g_{1},\dots,g_{p} \in \Hm\setminus \so0$.\\
   \lstinline|Output|&a Gröbner basis $\lst{g_{1},\dots,g_{p},\dots,g_{t}}$ for $\gen{g_{1},\dots,g_{p}}$.
 \end{tabular}
\]

\begin{buchbalgorithm}\label{genBuchbergeralg}
  \leavevmode
\begin{lstlisting}
local variables $S:\Hm$, $i,j,u:\NN$;
$t\gets p$;
repeat
  $u\gets t$;
  for $i$ from $1$ to $u$ do
    for $j$ from $i$ to $u$ do
      $S\gets\overline{\rS(g_{i},g_{j})}^{\lst{g_{1},\dots,g_{u}}}$ by Algorithms  (*@\ref{genSalg}@*) and (*@\ref{gendivalg}@*);(*@\label{genBuchbergeralg:7}@*)
      if $S\neq 0$ then
        $t\gets t+1$;
        $g_{t}\gets S$
      fi
    od
  od
until $t=u$
\end{lstlisting}
\end{buchbalgorithm}

This algorithm is almost the same algorithm as in the case where the base ring is a discrete field.  The modifications are in the definition of S-polynomials, in the consideration of the auto-S-polynomials, and in the division of terms (see Item~\eqref{mononmial-module-1} of Definition~\ref{mononmial-module}). In line~\ref{genBuchbergeralg:7}, the algorithm  may be sped up by computing the remainder w.r.t.\ $\lst{g_{1},\dots,g_{t}}$ instead of $\lst{g_{1},\dots,g_{u}}$ only.

\begin{remark} \label{remBuchb1}
If the algorithm terminates, then we can transform the obtained  Gröbner basis into a Gröbner basis $\lst{{g'_{1}},\dots,g'_{t'}}$ such that no term of an element~$g'_{j}$ lies in $\gen{\,\LT(g'_{k})\mathrel;k\ne j\,}$ by replacing each element of the Gröbner basis with a remainder of it on division by the other nonzero elements and
by repeating this process until it stabilises.  Such a Gröbner basis is called a \emph{pseudo-reduced Gröbner basis}.
\end{remark}

\subsection{The syzygy algorithm for terms}\label{sec:syzygy-algorithm}

This algorithm takes also place in Context~\ref{context} for $\R$.  Note however that the divisibility test is not used here; only the zero test is used.
It has the following goal.
 \[\makebox[\textwidth]{%
 \begin{tabular}{rl}
   \lstinline|Input|&terms $T_{1},\dots,T_{p}\in\Hm$.\\
   \lstinline|Output|
   &a generating system~$\lst{\rS_{i,j}}_{1\leq i\leq j\leq p,\LP(T_{j}) = \LP(T_{i})}$ for $\Syz(T_{1},\dots,T_{p})$.
 \end{tabular}}
\]

In this algorithm, $(\epsilon_{1},\dots,\epsilon_{p})$ is the canonical basis of $\R[\uX]^{p}$.

\begin{syzalgorithm}\label{gensyzygyalg}
  \leavevmode
\begin{lstlisting}
local variables $i,j:\so{1,\dots,p}$, $J:\text{subset of }\so{1,\dots,p}$,
                $a,b:\R$, $\alpha,\beta:\NN^{n}$;
for $i$ from $1$ to $p$ do
  $J\gets\sotq{j\geq i}{\LP(T_{j}) = \LP(T_{i})}$;
  for $j$ in $J$ do
    compute $b\uX^{\beta},a\uX^{\alpha}$ such that $\rS(T_{i},T_{j})=b\uX^{\beta}T_{i} - a\uX^{\alpha}T_{j}$
            by Algorithm (*@\ref{genSalg}@*);
    $\rS_{i,j}\gets b\uX^{\beta}\epsilon_{i}-a\uX^{\alpha}\epsilon_{j}$
  od
od
\end{lstlisting}
\end{syzalgorithm}

\subsection{Schreyer's syzygy algorithm}\label{sec:Syzygy-algorithm}

This algorithm takes also place in Context~\ref{context} for~$\R$.
It has the following goal.
 \[\makebox[\textwidth]{%
 \begin{tabular}{rl}
   \lstinline|Input|&a Gröbner basis $\lst{g_{1},\dots,g_{p}}$ for a submodule of~$\Hm$.\\
   \lstinline|Output|&a Gröbner basis~$\lst{u_{i,j}}_{1\leq i\leq j\leq p,\LP(g_{j}) = \LP(g_{i})}$ for $\Syz(g_{1},\dots, g_{p})$\\
         &\enskip w.r.t.\ Schreyer's monomial order induced by~$>$ and
         \(\lst{g_{1},\dots,g_{p}}\).
 \end{tabular}}
\]

In this algorithm, $(\epsilon_{1},\dots,\epsilon_{p})$ is the canonical basis of~$\R[\uX]^{p}$. 

\begin{Syzalgorithm}\label{genSyzygyalg}\leavevmode
\begin{lstlisting}
local variables $i,j:\so{1,\dots,p}$, $J:\text{subset of }\so{1,\dots,p}$,
                $a,b:\R$, $\alpha,\beta:\NN^{n}$, $q_{\ell}:\R[\uX]$;
for $i$ from $1$ to $p$ do
  $J\gets\sotq{j\geq i}{\LP(g_{j}) = \LP(g_{i})}$;
  for $j$ in $J$ do
    compute $b\uX^{\beta},a\uX^{\alpha}$ such that $\rS(g_{i},g_{j})=b\uX^{\beta}g_{i} - a\uX^{\alpha}g_{j}$
            by Algorithm (*@\ref{genSalg}@*);
    compute $q_{1},\dots,q_{p}$ such that (*@\label{genSyzygyalg:16}@*)
      $\rS(g_{i},g_{j})=q_{1}g_{1}+\dots+q_{p}g_{p}$ by Algorithm (*@\ref{gendivalg}@*) (*@\textrm{(note that}@*)
        $\LM(\rS(g_{i},g_{j})) \geq \LM(q_{\ell}) \LM(g_{\ell})\textrm{ whenever }q_{\ell}\ne0)$; (*@\label{genSyzygyalg:18}@*)
    $u_{i,j}\gets b\uX^{\beta}\epsilon_{i}-a\uX^{\alpha}\epsilon_{j}-q_{1}\epsilon_{1}-\dots-q_{p}\epsilon_{p}$
  od(*@\label{gensyzygyalg:19}@*)
od
\end{lstlisting}
\end{Syzalgorithm}

The polynomials $q_{1},\dots,q_{p}$ of lines~\ref{genSyzygyalg:16}--\ref{genSyzygyalg:18} may have been computed while constructing the Gröbner basis.

\begin{remark} \label{remSyzgen}
For an arbitrary system of generators $\lst{h_{1},\dots,h_{r}}$ for a  submodule $U$ of~$\Hm$, the syzygy module of $\lst{h_{1},\dots,h_{r}}$ is easily obtained from the syzygy module of a Gröbner basis for $U$ (see \cite[Theorem~296]{Y5}).
\end{remark}

\section{The
algorithms in the case of a valuation ring}\label{sec:case-valuation-rings}

This is the case of a local Bézout ring. We consider a coherent valuation ring $\V$ with a divisibility test.
In this case, we get simplified versions of the algorithms given in Section~\ref{sec:generic-algorithms}.
We recover  the algorithms given in~\cite{Ye2,Y5}, but for modules instead of ideals. In particular, we generalise
Buchberger's algorithm to convenient
 valuation rings and modules.
Note that the algorithm given in~\cite{Ye2}   contains a bug which is
corrected in the corrigendum~\cite{Y2} to the papers~\cite{HY,Ye2}.

\begin{divalgorithm}[see {\cite[Definition~226]{Y5}}]\label{def222}
  Let $\V $ be a valuation ring with a divisibility test.
  In the Division algorithm~\ref{gendivalg}, instead of defining the set~$D$ and finding the gcd~$d$, one may look out for the first~$\LT(h_{i})$ such that $\LT(h_{i})$ divides~$\LT(h')$; in case of success, the algorithm proceeds with this index~$i$, and the Bézout identity of line~\ref{gendivalg:5} is not needed.

\begin{lstlisting}
local variables $i:\so{1,\dots,p}$, $c:\R$, $h':\H_{m}$, notdiv${}:{}$boolean;
$q_{1} \gets 0$; ... ; $q_{p} \gets 0$; $r \gets 0$; $h' \gets h$;
while $h' \ne 0$ do
  $i \gets 1$;
  notdiv := true;
  while $i \leq p$ and notdiv do
    if $\LT(h_{i}) \divides \LT(h')$ then
      find $c$ such that $c \LC(h_{i}) = \LC(h')$;
      $q_{i} \gets q_{i} + c(\LM(h')/\LM(h_{i}))$;
      $h' \gets h' - c(\LM(h')/\LM(h_{i})) h_{i}$;
      notdiv := false
    else
      $i \gets i + 1$
    fi
  od;
  if notdiv then
    $r \gets r + \LT(h')$;
    $h' \gets h' - \LT(h')$
  fi
od
\end{lstlisting}
\end{divalgorithm}

\begin{spolyalgorithm}[see {\cite[Definition~229]{Y5}}]\label{def222Spaire}
  Let $\V$ be a coherent valuation ring.  We define the \emph{S-polynomial} of two nonzero vectors in~$\Hm$ by the S-polynomial algorithm~\ref{genSalg}. In this algorithm, the finding of $a,b$ in lines~\ref{genSalg:11}-\ref{genSalg:12} will take the following simple form, typical for valuation rings:
\begin{lstlisting}[numbers=none]
    find $a,b$ such that
      $a\LC(g)=b\LC(f)$ (*@\textrm{with $a=1$ or $b=1$}@*)
\end{lstlisting}
    This does not rely on the divisibility test: the explicit disjunction ``$a$ divides $b$ or $b$  divides $a$'' is sufficient.
When we have a divisibility test, the following expression arises for $\rS(f,g)$ with $f\neq g$, $\LP(f)=\LP(g)$, $\mdeg(f)=\mu$, $\mdeg(g)=\nu$:
\[\makebox[\textwidth]{$
  \rS(f,g)=
    \begin{cases}
    \hphantom b\uX^{(\nu-\mu)^{+}} f-a\uX^{(\mu-\nu)^{+}}g & \text{if $\LC(g)\divides\LC(f)$, where $\LC(f)=a\LC(g)$} \\
     b\uX^{(\nu-\mu)^{+}}f-\hphantom a\uX^{(\mu-\nu)^{+}}g & \text{otherwise, where $b\LC(f)=\LC(g)$.} \\
   \end{cases}$}
 \]
Note also that the annihilator $\Ann(\LC(f))$ appearing in the computation of the auto-S-polynomial is principal because $\V$ is a
coherent valuation ring: there is a $b$ such that $\Ann(\LC(f))=b\V$ ($b$~being defined up to a unit, see~\cite[Exercise IX-7]{ACMC}).
\end{spolyalgorithm}

\begin{example}[S-polynomials over {$\R=\mathbb{F}_{2}[Y]/\gen{Y^{r}}$}, $r\geq 2$, a generalisation of {\cite[Example~231]{Y5}}]\label{F2y}
The ring $\V=\mathbb{F}_{2}[Y ]/\gen{Y^{r}}= \mathbb{F}_{2}[y]$
(where $y = \overline Y $) is a zero-dimensional coherent  valuation ring
with nonzero zerodivisors ($\Ann(y^{k}) = \gen{y^{r-k}}$). Each nonzero element $a$ of this ring may be written in a unique way as $y^{k}(1+y b)$ with $k=0,\dots,r-1$ and $1+yb$ a unit.
Let $f \neq  g \in \V[\uX]\setminus \so{0}$ and $ \mu = \mdeg(f) $,
$\nu = \mdeg(g)
$.
If $\LC(g)=y^{k}(1+y b)$ and $\LC(f)=y^{\ell}(1+y c)$, then
\[
  \begin{aligned}
  \rS(f,g)&=
    \begin{cases}
    \uX^{(\nu-\mu)^{+}} f-(1+yc)(1+yb)^{-1}y^{\ell-k}\uX^{(\mu-\nu)^{+}}g & \text{if $k\leq \ell$} \\
    (1+yb)(1+yc)^{-1}y^{k-\ell}\uX^{(\nu-\mu)^{+}}f-\uX^{(\mu-\nu)^{+}}g & \text{if $k>\ell$} \\
   \end{cases}\\&\underset{\makebox[0pt]{\tiny up to a unit }}=\quad
    \begin{cases}
    (1+yb)\uX^{(\nu-\mu)^{+}} f-(1+yc)y^{\ell-k}\uX^{(\mu-\nu)^{+}}g & \text{if $k\leq \ell$} \\
    (1+yb)y^{k-\ell}\uX^{(\nu-\mu)^{+}}f-(1+yc)\uX^{(\mu-\nu)^{+}}g & \text{if $k>\ell$}. \\
   \end{cases}
 \end{aligned}
 \]
For the computation of the auto-S-polynomial~$\rS(f, f)$, two cases may
arise:
\begin{itemize}
\item If $\LC (f)$ is a unit, then $\rS(f, f) = 0$.
\item If $\LC (f)$ is $y^{k}$ ($k>0$) up to a unit, then $\rS(f, f) = y^{r-k}f$.
\end{itemize}
E.g., with $r=2$, using the lexicographic order for which $X_{2}>X_{1}$ and considering the polynomials $f=yX_{2}+X_{1}$ and
$g=yX_{1}+y$, we have:
\[\rS(f,g)=X_{1}f-X_{2}g=X_{1}^{2}+yX_{2},\enspace
  \rS(f,f)=yf=yX_{1},\enspace
  \rS(g,g)=yg=0\text.
\]
\end{example}

\section{Termination of Buchberger's algorithm for a Bézout ring}\label{case:Bezout-rings}

The following lemma provides a necessary and sufficient condition for a term to belong to a module generated by terms over a coherent strict Bézout ring with a divisibility test.

\begin{lemma}[Term modules, see~{\cite[Lemma~227]{Y5}}]\label{lemdivis}
Let \(\R\) be a coherent strict Bézout ring with a divisibility test. Let
\(U\) be a submodule of
\(\Hm\) generated by a finite collection of terms \(a_{\alpha}\uX^{\alpha} e_{i_{\alpha}}\) with \(\alpha \in A\).
A
term \(b \uX^{\beta} e_{r}\) lies in \(U\) iff there is a nonempty subset~$A'$ of~$A$ such that \(\uX^{\alpha} e_{i_{\alpha}}\) divides \(\uX^{\beta} e_{r}\) for every~$\alpha\in A'$ (i.e.\ \(i_{\alpha}=r\)  and \(\uX^{\alpha} \divides \uX^{\beta}\)) and  \(\gcd_{\alpha\in A'}(a_{\alpha})\) divides~\(b\). In the local case, there hence is an~\(a_{\alpha}\) with~\(\alpha\in A'\) that divides~\(b\).
\end{lemma}

\begin{proof}
  The condition is clearly sufficient.  For the
  necessity, write
  \begin{equation*}
    b \uX^{\beta} e_{r} = \som_{\alpha\in\tilde A}c_{\alpha}
    a_{\alpha}\uX^{\gamma_{\alpha}}\uX^{\alpha}
    e_{i_{\alpha}}
\end{equation*}
  with $\tilde A\subseteq A$, $c_{\alpha}\in \R \setminus
  \so{0}$, and $\uX^{\gamma_{\alpha}} \in
  \mathbb{M}_{n}$. Then $b=\sum_{\alpha\in A'} c_{\alpha}
  a_{\alpha}$, where $A'$ is the set of
  those~$\alpha$ such that ${\gamma_{\alpha}}+\alpha=
  \beta$ and $i_{\alpha}=r$. For each $\alpha\in A'$,
  $\uX^{\alpha}$ divides~$\uX^{\beta}$. Since the gcd of
  the~$a_{\alpha}$'s with $\alpha\in A'$ divides
  every~$a_{\alpha}$, it also divides~$b$.\qedhere
  \end{proof}

The following lemma is a key result
for the characterisation of Gröbner bases by means of
S-polynomials: see~\cite[Chapter~2, \S6, Lemma~5]{CLO}  and, for valuation rings, \cite[Lemma 233, adding the hypothesis of coherence]{Y5}.

\begin{lemma}\label{lemm}
Let \(\R\) be a coherent strict Bézout ring and \(f_{1},\dots,f_{p} \in \Hm\setminus\{0\}\) with the same leading monomial~\(M\). Let \(c_{1},\dots,c_{p} \in \R\). If \(c_{1}f_{1}+\cdots+c_{p}f_{p}\) vanishes or has leading monomial $<M$, then
\(c_{1}f_{1}+\cdots+c_{p}f_{p}\) is a linear combination with coefficients in
\(\R\) of the S-polynomials \(\rS(f_{i},f_{j})\) with \(1\leq i\leq j\leq p\).
\end{lemma}

\begin{proof}
  Let us write, for $j\ne i$,
  \(
  \LC(f_{j})=d_{i,j}a_{i,j}
  \)
  with $d_{i,j}=\gcd(\LC(f_{i}),\LC(f_{j}))$,
  so that $\gcd(a_{i,j},a_{j,i})=1$ and $\rS(f_{i},f_{j})=a_{i,j}f_{i}-a_{j,i}f_{j}$. For each permutation $i_{1},\dots,i_{p}$ of $1,\dots,p$,
  we shall transform the sum $a_{i_{1},i_{2}}\*\cdots a_{i_{p-1},i_{p}}\*(c_{1}f_{1}+\cdots+c_{p}f_{p})$ by replacing successively
  \[
    \begin{aligned}
      a_{i_{1},i_{2}}&f_{i_{1}}&&\text{by $\rS(f_{i_{1}},f_{i_{2}})+a_{i_{2},i_{1}}f_{i_{2}}$,}\\
      \vdots&&&\hskip6em\vdots\\
      a_{i_{p-1},i_{p}}&f_{i_{p-1}}&&\text{by $\rS(f_{i_{p-1}},f_{i_{p}})+a_{i_{p},i_{p-1}}f_{i_{p}}$.}
    \end{aligned}
  \]
  At the end, the sum will be a linear combination of $\rS(f_{i_{1}},f_{i_{2}})$, $\rS(f_{i_{2}},f_{i_{3}})$, \dots, $\rS(f_{i_{p-1}},f_{i_{p}})$, and $f_{i_{p}}$; let $z$ be the coefficient of $f_{i_{p}}$ in this combination. The sum as well as each of the S-polynomials vanish or have leading monomial $<M$, so that the hypothesis yields $z\LC(f_{i_{p}})=0$; therefore $zf_{i_{p}}$ is a multiple of $\rS(f_{i_{p}},f_{i_{p}})$.

  It remains to obtain a Bézout identity w.r.t.\ the products $a_{i_{1},i_{2}}\cdots a_{i_{p-1},i_{p}}$, because it yields an expression of $c_{1}f_{1}+\cdots+c_{p}f_{p}$ as a
  linear combination of the required
  form. For this, it is enough to develop the product of the
  $\tbinom s2$~Bézout identities w.r.t.~$a_{i,j}$ and~$a_{j,i}$,
  $1\leq i<j\leq p$: this yields a sum of products of
  $\tbinom s2$~terms, each of which is either~$a_{i,j}$ or~$a_{j,i}$,
  $1\leq i<j\leq p$, so that it is indexed by the tournaments on the
  vertices~$1,\dots,p$; every such product contains a product of the above
  form~$a_{i_{1},i_{2}}\cdots a_{i_{p-1},i_{p}}$ because every tournament
  contains a hamiltonian path (see
  \cite{redei35}). \end{proof}

\begin{remark}
  The above proof results from an analysis of the following proof in the case where $\R$ is local and $m=1$, which entails in fact the general case.
  Since $\R$ is a
valuation ring, we may consider a permutation $i_{1},\dots,i_{p}$ of $1,\dots,p$ such that $\LC(f_{i_{p}}) \divides \LC(f_{i_{p-1}}) \divides
\cdots \divides \LC(f_{i_{1}})$. Thus $\rS(f_{i_{1}},f_{i_{2}})= f_{i_{1}} - a_{i_{2},i_{1}}f_{i_{2}}$, \dots, $\rS(f_{i_{p-1}},f_{i_{p}})=f_{i_{p-1}}-a_{i_{p},i_{p-1}}f_{i_{p}}$ for
some~$a_{i_{2},i_{1}},\dots,a_{i_{p},i_{p-1}}$. Then, by replacing successively $f_{i_{k}}$ by $\rS(f_{i_{k}},f_{i_{k+1}})+a_{i_{k+1},i_{k}}f_{i_{k+1}}$, the linear combination $c_{1}f_{1}+\dots+c_{p}f_{p}$ may be rewritten as a linear combination of $\rS(f_{i_{1}},f_{i_{2}})$, \dots, $\rS(f_{i_{p-1}},f_{i_{p}})$, and $f_{i_{p}}$, with the coefficient of~$f_{i_{p}}$ turning out to lie in $\Ann(\LC(f_{i_{p}}))$.
\end{remark}

Lemma~\ref{lemm} enables us to generalise some
classical results on the existence and characterisation of
Gröbner bases  to the case of coherent
strict Bézout  rings with a divisibility test. See~\cite[Theorem 234]{Y5} for the case of valuation rings and ideals.

\begin{theorem}[Buchberger's criterion for Gröbner bases]\label{thm1}
  Let \(\R\) be a coherent
 strict Bézout ring with a divisibility test and
 \(U = \gen{g_{1},\dots,g_{p}}\) a nonzero submodule of~\(\Hm\). Then \(G =
\lst{g_{1},\dots,g_{p}} \) is a Gröbner basis for \(U\) iff the remainder of~\(\rS(g_{i},g_{j})\) on division by
\(G\) vanishes for
all pairs \(i \leq j\). \end{theorem}

Theorem~\ref{thm1} entails that Buchberger's algorithm~\ref{genBuchbergeralg} constructs a Gröbner basis for finitely generated ideals of coherent valuation rings with a divisibility test when such a basis exists (compare~\cite[Algorithm 235]{Y5}). The two following theorems provide a general explanation for the termination of Buchberger's algorithm and are therefore pivotal.

\begin{theorem}[Termination of Buchberger's algorithm, case $m=1$]\label{Grob-Buch} Let \(\V\) be a coherent
 valuation ring with a divisibility test,
\(I\) a nonzero finitely generated ideal of \(\V[\uX]\), and
\(>\) a monomial order on \(\V[\uX]\). If \(\LT(I)\) is finitely generated, then Buchberger's algorithm~\ref{genBuchbergeralg} computes a finite Gröbner basis for~\(I\).
\end{theorem}

\begin{proof}
  Let $f_{1},\dots,f_{p}\in\V[\uX]\setminus\{0\}$ be generators of~$I$. Let $ \LT(I)=\langle\LT(g_{1}),\allowbreak\dots,\LT(g_{r})\rangle$  with $g_{i} \in I\setminus\{0\}$.
Let $1 \leq k \leq r$.  As $g_{k} \in I$, there exist~$E\subseteq\{1,\dots,p\}$ and $h_i\in \V[\uX]\setminus\{0\}$, $i\in E$, such that
\begin{equation}
g_{k}=\som_{i\in E} h_{i}  f_{i}\label{eqlospol}
\end{equation}
with $\mdeg(g_{k})\leq \sup_{i\in E}(\mdeg(M_{i}N_{i}))\eqqcolon\gamma$ (we call it the \emph{multidegree of the expression}~\eqref{eqlospol}
for~$g_{k}$ w.r.t.\ the generating set $\so{f_{1},\dots,f_{p}}$ of $I$), where $M_{i}=\LM (h_{i})$ and $N_{i}=\LM (f_{i})$. Let $F=\{\,i\in E\mathrel;\mdeg(M_{i}N_{i})=\gamma\,\}$.

\begin{enumerate}[label=Case~\arabic*:]

\item $\mdeg(g_{k})= \gamma$, say  $\mdeg(g_{k})= \mdeg(M_{i_{0}}N_{i_{0}})$
for some $i_{0}\in F$. As the leading coefficients of the
$h_{i}f_{i}$'s with $i\in F$ are
comparable w.r.t.\ division, we can suppose that all of them are
divisible by the leading coefficient of
 $h_{i_{0}}f_{i_{0}}$. It follows that $\LT (g_{k}) \in \gen{\LT(f_{i_{0}})} \subseteq \gen{\LT(f_{1}),\dots,\LT(f_{p})}$.

\item $\mdeg(g_{k}) <\gamma$. We have
\[
  \begin{aligned}
    g_{k}&=\som_{i\notin F}h_{i}f_{i}+\som_{i\in F}h_{i}f_{i}\\
    &=\som_{i\notin F}h_{i}f_{i}+\som_{i\in F}(h_{i}-\LT (h_{i}))f_{i}+\som_{i\in F}\LT (h_{i})f_{i}.
\end{aligned}
\]
Letting $c_{i}=\LC (h_{i})$, we get
\[
  \mdeg\big(\som_{i\in F}c_{i} M_{i}f_{i}\big) < \gamma.
\]
By virtue of Lemma~\ref{lemm}, there exists a finite family $(a_{i,j})$ of
elements of~$\V$ such that
\[
  \som_{i\in F}c_{i} M_{i}f_{i}= \som_{i\leq j \in F}a_{i,j} \rS(M_{i}f_{i},M_{j}f_{j}).
\]
But, for $i\leq j \in F$, letting $N_{i,j}=\lcm(N_{i},N_{j})$ and writing $\rS(f_{i},f_{j})=a \frac{N_{i,j}}{N_{i}}f_{i}+b \frac{N_{i,j}}{N_{j}}f_{j}$ for some $a,\,b \in \V$,
we have $\rS(M_{i}f_{i},M_{j}f_{j})= a \frac{X^{\gamma}}{M_{i}N_{i}}M_{i}f_{i}+b \frac{X^{\gamma}}{M_{j}N_{j}}M_{j}f_{j}=\frac{X^{\gamma}}{N_{i,j}}\, \rS(f_{i},f_{j})$. It follows that
\[
  \som_{i\in F}c_{i} M_{i}f_{i}
=\som_{i\leq j \in F}a_{i,j} m_{i,j} \rS(f_{i},f_{j}),
\]
where the $m_{i,j}$'s  are monomials. Thus we obtain another expression for $g_{k}$,
\[\makebox[\textwidth]{$
    \displaystyle g_{k}=\som_{i\notin F}h_{i}f_{i}+\som_{i\in F}(h_{i}-\LT (h_{i}))f_{i}+\som_{i\leq j \in F}a_{i,j} m_{i,j} \rS(f_{i},f_{j})\text,
$}\]
and the multidegree of this expression, now w.r.t.\ the generating set of~$I$ obtained by adding the elements $\rS(f_{i},f_{j})$, $i \leq j \in F$, to the $f_{1},\dots,f_{p}$, 
is  $ < \gamma$. Reiterating this, we end up with a situation like that of Case 1 for all the $g_{k}$'s because the set of monomials is
well-ordered. So we reach the termination condition in Algorithm~\ref{genBuchbergeralg} after a finite number of steps.\qedhere
\end{enumerate}
\end{proof}

\begin{theorem}[Termination of Buchberger's algorithm]\label{Grob-Buch2} Let \(\R\) be a coherent  strict Bézout ring
with a divisibility test and
\(U \) a nonzero finitely generated submodule of~\(\Hm\). If \(\LT(U)\) is finitely generated, then Buchberger's algorithm~\ref{genBuchbergeralg} computes a Gröbner basis for~\(U\).
\end{theorem}

\begin{proof} It suffices to prove the result when $\R$ is local and $m=1$, in which case this is Theorem~\ref{Grob-Buch}. Let us explain in a few words how to pass from the local to the global case (compare~\cite[Section~3.3.11]{Y5} and~\cite{GY}). Suppose that we are computing $\rS(f,g)$ and that the
leading coefficients $a$ and $b$ of $f$ and $g$ are uncomparable under division. A key fact is
that if we write $a = \gcd(a,b)\,a'$, $b = \gcd(a, b)\,b'$ with
$\gcd(a',b')=1$, then $a$ divides $b$ in
$\R[\frac{1}{a'}]$, $b$ divides $a$ in
$\R[\frac{1}{b'}]$, and the two multiplicative subsets
$a'^{\mathbb{N}}$ and $b'^{\mathbb{N}}$ are comaximal
because $1\in \gen{a',b'}$. Then $\R$
splits into $\R[\frac{1}{a'}]$ and
$\R[\frac{1}{b'}]$, and we can continue as if $\R$
were a valuation ring. If $\mdeg(f)= \mu$ and $\mdeg(g)=\nu$, then $\rS(f,g)$ is
being computed as follows:
\begin{itemize}
\item in the ring $\R[\frac{1}{b'}]$, $\rS(f,g)=
X^{(\nu-\mu)^{+}}f- \frac{a'}{b'}
X^{(\mu-\nu)^{+}}g\eqqcolon S_{1}$;
\item in the ring $\R[\frac{1}{a'}]$, $\rS(f,g)=
  \frac{b'}{a'}X^{(\nu-\mu)^{+}}f-
X^{(\mu-\nu)^{+}}g\eqqcolon S_{2}$.
\end{itemize}
But, letting $\rS\coloneqq b'X^{(\nu-\mu)^{+}}f- a'
X^{(\mu-\nu)^{+}}g$, we have
\[
\rS=b'S_{1}=a'S_{2}\text.\]
As $\rS$ is equal to $S_{1}$ up to a unit
in  $\R[\frac{1}{b'}]$, and to  $S_{2}$ in
$\R[\frac{1}{a'}]$, it can replace both of them, and thus
there was no need to open the two branches
$\R[\frac{1}{a'}]$ and $\R[\frac{1}{b'}]$.
\end{proof}

\subsection*{When is a valuation ring a Gröbner ring?}\label{sec:valgrob}
\addcontentsline{toc}{subsection}{When is a valuation ring a Gröbner ring?}

We recall here some results given in \cite{Y5} on the interplay between the concepts of Gröbner ring, Krull dimension, and archimedeanity; here are the relevant definitions. 

\begin{definition}\leavevmode
  \begin{itemize}
  \item The (Jacobson) \emph{radical} $\Rad(\R)$ of an arbitrary ring
    $\R$ is the ideal $\sotq{a\in\R}{1+a\R\subseteq \R^{\times}}$,
    where $\R^{\times}$~is the unit group of $\R$.
  \item The \emph{residual field} of a local ring
    $\R$ is the quotient $\R/\Rad(\R)$. The local ring~$\R$ is \emph{residually discrete} if its residual field is discrete: this means that $x\in\R^{\times}$ is decidable.
    A nontrivial local ring~$\R$ is residually discrete iff it is the disjoint union of $\R^{\times}$ and $\Rad(\R)$.
  \item A residually discrete valuation  ring $\R$ is \emph{archimedean} if
\[
  \forall a,b \in \Rad(\R)\setminus \so{0}\ \exists k\in \NN\ \ a\divides b^{k}\text.
\]
\item A strongly discrete ring~$\R$ is a \emph{Gröbner ring} if for every $n \in \mathbb{N}$ and every finitely generated ideal~$I$ of $\R[\uX]$ endowed with the lexicographic monomial order, the module~$\LT (I)$ is finitely generated as well.
\end{itemize}
\end{definition}

One sees easily that a Gröbner ring is coherent (\cite[Proposition~224]{Y5}).
Moreover if $\R$ is Gröbner, then so is $\R[Y]$.

For a coherent valuation ring with a divisibility test, it is proved in \cite{Y5} that archimedeanity is equivalent to being a Gröbner ring  (at least when we assume that there is no nonzero zerodivisor or there exists a nonzero zerodivisor, see \cite[Theorem~272]{Y5}). For a valuation domain with a divisibility test, it is proved that the condition is equivalent to having Krull dimension~$\leq 1$ (\cite[Theorem~256]{Y5}). This implies that a strongly discrete Prüfer domain is Gröbner iff it has Krull dimension~${\leq 1}$ (\cite[Corollary~6]{Y4}). This applies to Bézout domains with a divisibility test. When a coherent valuation ring with a divisibility test has a nonzero zerodivisor, it is proved that 
archimedeanity is equivalent to being zero-dimensional (\cite[Proposition~265]{Y5}).

Let us now, for the comfort of the reader, provide simple arguments for some of these results. 
Recall that a ring $\V$ has Krull dimension~$\leq 1$ if, given $a,b\in\R$,
\begin{equation}
  \exists k,\ell\in\NN\ \exists x,y\in\R\ \ b^{\ell}(a^{k}(ax-1)+by)=0\text;\label{(*)}
\end{equation}
when $b$ is regular and $a\in\Rad(\V)$, we get that $a^k=zb$ for some $k$ and some $z$. 
This shows that a valuation domain of Krull dimension~$\leq 1$ is archimedean.
Conversely, an equality $a^k=zb$ is a particularly simple case of \eqref{(*)} (take $x=0$). Also, when $a$ is invertible, one has $ax-1=0$ for some~$x$, which is also a form of \eqref{(*)}. So, if in a local ring the disjunction ``$x$ is invertible or $x\in\Rad(\V)$'' is explicit (i.e.\ if the residual field is discrete), then archimedeanity implies Krull dimension~$\leq 1$.   
Summing up, an archimedean valuation ring with a divisibility test has Krull dimension~$\leq 1$, and a valuation domain with Krull dimension~$\leq 1$
is archimedean: so a valuation domain is archimedean iff it has Krull dimension~$\leq 1$.

Recall now that for a local ring, being zero-dimensional means that every element is invertible or nilpotent. Let us consider a valuation ring with a divisibility test containing a nonzero zerodivisor~$x$.  We have $xy=0$ with $y\neq 0$. If $x=yz$, then $y^2z=0$, so that $x^2=0$. If $y=xz$, then $y^2=0$. So we have a nonzero nilpotent element $u$. In this case archimedeanity is equivalent to being zero-dimensional.  Indeed, assume first archimedeanity. For an $a\in\Rad(\V)$, we have $u\divides a^k$, so $a^{2k}=0$. Then assume zero-dimensionality. For any $a,b\in\Rad(\V)$, we have a~$k$ such that $a^k=0$, so $b\divides a^k$.

So, for a coherent valuation ring with a divisibility test, if $0$ is the unique zerodivisor, archimedeanity is equivalent to having  dimension $\leq 1$, and if $\V$ has a nonzero zerodivisor, archimedeanity is equivalent to being zero-dimensional.

Now assume that $\V$ is a coherent valuation ring with a divisibility test.
We first compute $(c:d)$ when $c,d\neq 0$. We note that $(c:d)=\gen{u}$
for some $u$ (since it is finitely generated). If $c\divides d$, then $(c:d)=\gen{1}$.
If $d\divides c$, then we have a $y$ with $c=dy$. So $y\in\gen{u}$, say $y=tu$. Since $u\in(c:d)$, we have a $z$ with $du=cz=dyz=dutz$. So $du(1-tz)=0$. If $1-tz$ is invertible, then $du=0$, so that $c=dy=dut=0$, which is impossible. So $tz$ is invertible and $\gen{u}=\gen{y}$: more precisely $u=yt'$ with $t'$ invertible.

Now let $a,b\in\Rad(\V)\setminus\{0\}$. We show that $(b:a^\infty)$ is finitely generated iff
$b\divides a^k$ for some~$k$. If $a^k=bx$ then $(b:a^k)=\gen{1}$, so $(b:a^\infty)=\gen{1}$. If $(b:a^\infty)$ is finitely generated, then we have a~$k$ such that $(b:a^k)=(b:a^{k+1})$. If $b\divides a^k$ or $b\divides a^{k+1}$, then we are done.
The other case ($a^k\divides b$ and $a^{k+1}\divides b$) is impossible, for if we have $x,y$ such that $b=a^kx=a^{k+1}y=a^k(ay)$, then
$$\gen{y}=(b:a^{k+1})=(b:a^k)=\gen{x}=\gen{ay},
$$ 
so that $y=uay$ and $(1-ua)y=0$ for some $u$; since $a\in\Rad(\V)$, $1-ua$ is invertible, so that $y=0$,
which implies $b=0$, a contradiction.

We have shown that $\V$ is archimedean iff $(b:a^\infty)$ is finitely generated for all $a,b\in\Rad(\V)\setminus\{0\}$. 

We note also that for an arbitrary commutative ring $\R$, one has  
$$\forall a,b\in\R\;\gen{1+b\+Y,a}\cap\R= (b:a^\infty)\text. 
$$
So a coherent valuation ring~$\V$ with a divisibility test is archimedean iff the ideal $\gen{1+b\+Y,a}\cap\V$ is finitely generated for all $a,b\in\R$.
This condition is fulfilled as soon as $\V$ is $1$-Gröbner (i.e.\ satisfies the definition of Gröbner rings with $n=1$).

For other details on this topic see \cite[Exercise 372 p.~207, solution p.~221, Exercise 387 p.~218, solution p.~251]{Y5}.

\section{The syzygy theorem and Schreyer's algorithm for a valuation  ring}\label{sec5}

In the book \emph{Gröbner bases in commutative algebra}, Ene and Herzog propose the following exercise.

\begin{problem}[{\cite[Problem 4.11, p.~81]{EH}}]\label{claim-false}
  Let $>$ be a    monomial order on the free $S$-module $F=\bigoplus_{j=1}^{m}Se_{j}$ [where $S=\K[\uX]$
with $\K$ a discrete field], let $U\subset F$ be a submodule of $F$, and suppose that $ \LT (U)=\bigoplus_{j=1}^{m}I_{j}e_{j}$. Show that $U$ is a free $S$-module iff $I_{j}$ is a principal ideal for $j=1,\dots,m$.
\end{problem}

 It is obvious that this condition is sufficient. Unfortunately, it is not necessary as shows the following example, so that
 the statement of~\cite[Problem 4.11]{EH} is not correct.

 \begin{example}\label{counter}
   Let $>$ be a   TOP   monomial order on  $\K[\X,\X[2]]^{2}$ for which $\X[2]>\X$, $\K$ being a field, let $e_{1}=(1,0)$ and $e_{2}=(0,1)$, and consider the free submodule~$U$   of $\K[\X,\X[2]]^{2}$ generated by $u_{1}=(\X[2],\X)$ and $u_{2}=(\X,0)$. Then $\LT (u_{1})=\X[2]e_{1}$, $\LT (u_{2})=\X e_{1}$, $\rS(u_{1},u_{2})= \X u_{1}-\X[2]u_{2}=\X^{2}e_{2} \eqqcolon u_{3}$, and $\rS(u_{1},u_{3})=\rS(u_{2},u_{3})=0$. It follows that  $(u_{1},u_{2},u_{3})$ is a Gröbner basis for $U$, and  $ \LT (U)=\gen{\X[2],\X}e_{1}\oplus \gen{\X^{2}}e_{2}$. One can see that  $\gen{\X[2],\X}$ is not principal and $\LT(U)$ is not free, while $U$ is free.
\end{example}

So we content ourselves with the following observation.

\begin{remark}\label{claim-rem}
  Let $>$ be a    monomial order on the free $S$-module $F=\bigoplus_{j=1}^{m}Se_{j}$, where $S=\V[\uX]$ and
$\V$ is a valuation domain. Let $U$ be a submodule of $F$ and suppose that $ \LT (U)=\bigoplus_{j=1}^{m}I_{j}e_{j}$, where $I_{j}$ is a principal ideal for $j=1,\dots,m$. Then $\LT(U)$ and $U$ are free $S$-modules. (Of course, this is not true anymore if $\V$ is a valuation ring with nonzero zerodivisors. Consider e.g.\ the ideal $U=\gen{8X+2}$ in  $(\ZZ /16\+\ZZ )[X]$: we have $ \LT (U)=\gen{2}$ (so that it is principal), but $U$ is not free since $8\+U=\gen{0}$.)
\end{remark}

We shall need the following proposition, which generalises \cite[Theorem 291]{Y5} to the case of modules.

\begin{proposition}[Generating set for the  syzygy module of a list of terms for a coherent valuation ring]\label{valuation1}
  Let \(\V\) be a coherent valuation ring,  \(\Hm\)  a free \(\V[\uX]\)-module with basis \((e_{1},\dots,e_{m})\),
  and terms~\(T_{1},\dots,T_{p}\) in~\(\Hm\). Considering the canonical basis \((\epsilon_{1},\dots,\epsilon_{p})\) of \(\V[\uX]^{p}\), the syzygy module \(\Syz(T_{1},\dots,
T_{p})\) is generated by the
\[
  \rS_{i,j}\in \V[\uX]^{p}\text{ with \(1\leq i \leq j \leq p\) and \(\LP(T_{i})=\LP(T_{j})\),}
\]
as computed by the Syzygy algorithm for terms~\ref{gensyzygyalg}.
\end{proposition}

Note that in the Syzygy algorithm for terms~\ref{gensyzygyalg}, the $a,b$ will be found as in
the S-polynomial algorithm~\ref{def222Spaire}, so that we get
\begin{equation}
  \rS_{i,j}=
  \begin{cases}
    b\epsilon_{i}&\text{if $i=j$, where $\gen{b}=\Ann(\LC(T_{i}))$,}\\
    \uX^{\beta}\epsilon_{i}-a\uX^{\alpha}\epsilon_{j} & \text{if $i<j$ and $\LC(T_{i})=a\LC(T_{j})$, else} \\
     b\uX^{\beta}\epsilon_{i}-\uX^{\alpha}\epsilon_{j} & \text{if $i<j$ and  $b\LC(T_{i})=\LC(T_{j})$.}
   \end{cases}\label{eq:1}
\end{equation}
Here $\beta=(\mdeg(T_{j})-\mdeg(T_{i}))^{+}$ and $\alpha=(\mdeg(T_{i})-\mdeg(T_{j}))^{+}$.

Now we shall follow closely \foreignlanguage{german}{Schreyer}'s ingenious proof~\cite{Sc} of Hilbert's syzygy theorem via Gröbner bases, but with a valuation ring instead of a field. \foreignlanguage{german}{Schreyer}'s proof is very well explained in~\cite[\S\S~4.4.1--4.4.3]{EH}.

\begin{theorem}[Schreyer's algorithm for a coherent valuation ring with a divisibility test]\label{schvaluation1}
  Let \(\V\) be a coherent valuation ring with a divisibility test. Let \(U\) be a submodule of \(\Hm\) with Gröbner basis \(\lst{g_{1},\dots,g_{p}}\).  Then the relations~\(u_{i,j}\) computed by \foreignlanguage{german}{Schreyer}'s syzygy algorithm~\ref{genSyzygyalg} form a Gröbner basis for the syzygy module \(\Syz(g_{1},\dots, g_{p})\) w.r.t.\ \foreignlanguage{german}{Schreyer}'s monomial order induced by~\(>\) and \(\lst{g_{1},\dots,g_{p}}\). Moreover, for \(1 \leq i \leq j \leq p\) such that \(\LP(g_{i})=\LP(g_{j})\),
\begin{equation}\label{sch}
\LT (u_{i,j})=
\begin{cases}
    b\epsilon_{i} & \text{if \(i=j\), with \(\gen{b}=\Ann(\LC(g_{i}))\),}\\
    \uX^{\beta}\epsilon_{i} & \text{if \(i<j\) and \(\LC(g_{j})\divides\LC(g_{i})\), else} \\
    b\uX^{\beta}\epsilon_{i} & \text{if \(i<j\) and \(b\LC(g_{i})=\LC(g_{j})\),}
\end{cases}
\end{equation}
with \(\beta=(\mdeg(g_{j})-\mdeg(g_{i}))^{+}\).
\end{theorem}

\begin{proof}[Proof \textup{(a slight modification of the proof of~\cite[Theorem~4.16]{EH})}] Let us use the notation of \foreignlanguage{german}{Schreyer}'s sy\-zy\-gy algorithm~\ref{genSyzygyalg}. 
Let $1 \leq  i=j  \leq p$. As
$\LM(q_{\ell}) \LM(g_{\ell}) \leq \LM(\rS(g_{i}, g_{i})) < \LM (g_{i})$ whenever $q_{\ell}\ne0$, we infer that  $\LT (u_{i,i})=b\epsilon_{i}$ with $\gen{b}=\Ann(\LC(g_{i}))$.

Let $1 \leq  i <  j \leq p$ such that  $\LP(g_{i})=\LP(g_{j})$. Suppose that $\LC(g_{i})=a\LC(g_{j})$ for an~$a$: as $\LM (\uX^{\beta} g_{i})=\LM (a\uX^{\alpha} g_{j})$ and $i<j$, $\LT(\uX^{\beta}\epsilon_{i}-a\uX^{\alpha}\epsilon_{j})=\uX^{\beta}\epsilon_{i}$ w.r.t.\ \foreignlanguage{german}{Schreyer}'s monomial order induced by~\(>\), and because
$\LM(q_{\ell}) \LM(g_{\ell}) \leq \LM(\rS(g_{i}, g_{j})) < \LM (\uX^{\beta} g_{i})$ whenever $q_{\ell}\ne0$, we infer that $\LT (u_{i,j})=
\uX^{\beta}\epsilon_{i}$;
otherwise, with $b$ such that $b\LC(g_{i})=\LC(g_{j})$, we obtain similarly $\LT (u_{i,j})=
b\uX^{\beta}\epsilon_{i}$.

Let Equation~\eqref{eq:1} hold with $T_{\ell}=\LT(g_{\ell})$: then  $\LT(u_{i,j})=\LT(\rS_{i,j})$ holds for all $1\leq i\leq j\leq p$.

Let us show now that the relations $u_{i,j}$ form a Gröbner basis for the syzygy module $\Syz(g_{1},\dots, g_{p})$. For this, let $v = \sum_{\ell=1}^{p} v_{\ell} \epsilon_{\ell} \in \Syz(g_{1},\dots, g_{p})$ and let us show that there exist $1 \leq i \leq j \leq p$ with $\LP(g_{i})=\LP(g_{j})$ such that $\LT (u_{i,j})$ divides $\LT (v)$.
Let us write $\LM (v_{\ell} \epsilon_{\ell})=N_{\ell} \epsilon_{\ell}$ and $\LC (v_{\ell} \epsilon_{\ell})=c_{\ell}$ for $1 \leq \ell \leq p$. Then $\LM (v)= N_{i} \epsilon_{i}$ for some $1 \leq i \leq p$. Now let $v' = \sum_{\ell\in\mathcal{S}} c_{\ell}  N_{\ell}\epsilon_{\ell}$, where $\mathcal{S}$ is the set of those $\ell$ for which $N_{\ell}\LM (g_{\ell})= N_{i} \LM (g_{i})$. By definition of \foreignlanguage{german}{Schreyer}'s monomial order, we have $\ell \geq i$ for all $\ell \in \mathcal{S}$.
Substituting each $\epsilon_{\ell}$ in $v'$ by $T_{\ell}$, the sum becomes zero. Therefore $v'$ is a relation of the terms $T_{\ell}$ with $\ell \in \mathcal{S}$. By virtue of Proposition~\ref{valuation1}, $v'$ is an $\V[\uX]$-linear combination of the~$\rS_{\ell,j}$ with $\ell\leq j$ in $\mathcal{S}$. Taking into consideration Equation~\eqref{eq:1}, we infer, by virtue of Lemma~\ref{lemdivis}, that $\LT (v')$ is a multiple of $\LT(\rS_{i,j})$ for some $j\in\mathcal S$. The desired result follows since $\LT (v)=\LT (v')$.
\end{proof}

As a consequence of Theorem~\ref{schvaluation1}, we obtain the following constructive versions of Hilbert's syzygy theorem for a   valuation domain.

\begin{theorem}[Syzygy theorem for a valuation domain with a divisibility test]\label{hilbvaluation1}
Let \(M=\Hm/U\) be a finitely presented  \(\V[\uX]\)-module, where \(\V\) is a valuation domain with a divisibility test.
Assume that, w.r.t.\ {\color{OliveGreen}some} monomial order, $\LT(U)$ is finitely generated. Then \(M\) admits
a free \(\V[\uX]\)-resolution
\[
  0 \rightarrow F_{p} \rightarrow F_{p-1} \rightarrow \cdots \rightarrow F_{1} \rightarrow F_{0} \rightarrow M \rightarrow 0
\]
of length \(p \leq n+1\).

\end{theorem}

\begin{proof}
It suffices to prove that $U$ has
a free $\V[\uX]$-resolution of length $p \leq n$. Let $ (g_{1},\dots,g_{p})$ be a  Gröbner basis for $U$ w.r.t.\ the {\color{OliveGreen}considered monomial order. We can reorder the $g_j$'s so that whenever $\LM (g_{i})$ and $\LM (g_{j})$ involve the same basis element for
some $i<j$, say $\LM (g_{i})=N_{i} \epsilon_{k}$ and  $\LM (g_{j})=N_{j} \epsilon_{k}$, then $\deg_{X_n}(N_{i}) \ge \deg_{X_n}(N_{j})$. It follows} that the indeterminate~$X_n$ cannot appear in the leading terms of the $u_{i,j}$'s in~\eqref{sch}.  Thus, after at most $n$ computations of the iterated syzygies, we reach a situation where none of the indeterminates  $X_{n},\dots,X_{1}$ appears in the leading terms of the computed Gröbner basis for the iterated syzygy module. This implies that the iterated syzygy module is free (as noted in Remark~\ref{claim-rem}).
\end{proof}


\begin{corollary}[Syzygy theorem for a valuation domain of Krull dimension~$\leq 1$ with a divisibility test]\label{hilbvaluation0}
Let \(M=\Hm/U\) be a finitely presented  \(\V[\uX]\)-module, where \(\V\) is a valuation domain of Krull dimension~$\leq 1$ with a divisibility test. Then \(M\) admits
a finite free \(\V[\uX]\)-resolution
\[
  0 \rightarrow F_{p} \rightarrow F_{p-1} \rightarrow \cdots \rightarrow F_{1} \rightarrow F_{0} \rightarrow M \rightarrow 0
\]
of length \(p \leq n+1\).

\end{corollary}

\begin{example}\label{ex01}Let $g_{1}=\X[2]^{4}-\X[2],\,g_{2}=2\+\X[2],\,g_{3}=\X^{3}-1 \in \ZZ _{2\ZZ }[\X,\X[2]]$, and let us use the lexicographic order
$>_{1}$ for which $\X[2]>_{1}\X$. We have
\[
  \begin{aligned}
    \rS(g_{1},g_{2})&=2\+g_{1}-\X[2]^{3}g_{2}= -2\+\X[2]=-g_{2}\text,\\
    \rS(g_{1},g_{3})&=\X^{3}g_{1}-\X[2]^{4}g_{3}= \X[2]^{4}-\X[2]\X^{3}=g_{1}-\X[2]g_{3}\text,\\
    \rS(g_{2},g_{3})&=\X^{3}g_{2}-2\+\X[2]g_{3}= 2\+\X[2]=g_{2}\text.
  \end{aligned}
\]
Thus $(g_{1},g_{2},g_{3})$ is a (pseudo-reduced) Gröbner basis for $I=\gen{g_{1},g_{2},g_{3}}$ and $\LT (I)=\gen{\X[2]^{4},2\+\X[2],\X^{3}}$. By Theorem~\ref{schvaluation1},
$u_{1,3}=\lst{\X^{3}-1,0,-\X[2]^{4}+\X[2]}$, $u_{1,2}=\lst{2,-\X[2]^{3}+1,0}$, $u_{2,3}=\lst{0,\X^{3}-1,-2\+\X[2]}$ form a (pseudo-reduced) Gröbner basis for the syzygy module $\Syz(g_{1},g_{2},g_{3})$ w.r.t.\ \foreignlanguage{german}{Schreyer}'s monomial order $>_{2}$ induced by $>_{1}$ and  $\lst{g_{1},g_{2},g_{3}}$. In particular,
\[
  \begin{aligned}
    \LT (\Syz(g_{1},g_{2},g_{3}))&=\gen{\LT (u_{1,3}),\LT (u_{1,2}), \LT (u_{2,3})}\\
    &= \gen{\X^{3}\epsilon_{1},2\epsilon_{1},\X^{3}\epsilon_{2}}= \gen{2,\X^{3}}\epsilon_{1} \oplus \gen{\X^{3}}\epsilon_{2}\text,
  \end{aligned}
\]
where $(\epsilon_{1},\epsilon_{2},\epsilon_{3})$ stands for the canonical basis of $\ZZ _{2\ZZ }[\X,\X[2]]^{3}$. We have
\[\makebox[\textwidth]{$
    \begin{aligned}
      \rS(u_{1,3},u_{1,2})&=2\+u_{1,3}-\X^{3}u_{1,2}
      \begin{aligned}[t]
        &=(-2,\X[2]^{3}\X^{3}-\X^{3},-2\+\X[2]^{4}+2\+\X[2])\\
        &=-u_{1,2}+(\X[2]^{3}-1)u_{2,3}\text,
      \end{aligned}\\
      \rS(u_{1,3},u_{2,3})&=\rS(u_{1,2},u_{2,3})=0\text.
    \end{aligned}
$}\]
We recover that  $\lst{u_{1,3},u_{1,2},u_{2,3}}$ is a Gröbner basis for $\Syz(g_{1},g_{2},g_{3})$. By Theorem~\ref{schvaluation1}, the element $u_{1,3;1,2}=\lst{2,-\X^{3}+1,-\X[2]^{3}+1}$ forms a (pseudo-reduced) Gröbner basis for the syzygy module $\Syz(u_{1,3},u_{1,2},u_{2,3})$ w.r.t.\ \foreignlanguage{german}{Schreyer}'s monomial order~$>_{3}$ induced by $>_{2}$ and  $\lst{u_{1,3},u_{1,2},u_{2,3}}$. In particular, $\LT (\Syz(u_{1,3},u_{1,2},u_{2,3}))=\gen{\LT (u_{1,3;1,2})}= \gen{2}\epsilon'_{1} $,
where $(\epsilon'_{1},\epsilon'_{2},\epsilon'_{3})$ stands for the canonical basis of $\ZZ _{2\ZZ }[\X,\X[2]]^{3}$. By Remark~\ref{claim-rem}, $\Syz(u_{1,3},u_{1,2},u_{2,3})$ is free. We conclude that  $I$ admits
the following length-$2$ free $\ZZ _{2\ZZ }[\X,\X[2]]$-resolution:
\[
  \makebox[\textwidth]{$
    0 \longrightarrow \ZZ _{2\ZZ }[\X,\X[2]] \xrightarrow{u_{1,3;1,2}} \ZZ _{2\ZZ }[\X,\X[2]]^{3} \xrightarrow{
      \left(\begin{smallmatrix}
          u_{1,3}\\
          u_{1,2}\\
          u_{2,3}
        \end{smallmatrix}\right)} \ZZ _{2\ZZ }[\X,\X[2]]^{3} \xrightarrow{
      \left(\begin{smallmatrix}
          g_{1}\\
          g_{2}\\
          g_{3}
        \end{smallmatrix}\right)} I \rightarrow 0
$.}\]
It follows that  $\ZZ _{2\ZZ }[\X,\X[2]]/I$ admits
the following  length-$3$ free $\ZZ _{2\ZZ }[\X,\X[2]]$-resolution:
\[\makebox[\textwidth]{$
    0 \rightarrow \ZZ _{2\ZZ }[\X,\X[2]] \rightarrow\ZZ _{2\ZZ }[\X,\X[2]]^{3} \rightarrow\ZZ _{2\ZZ }[\X,\X[2]]^{3} \rightarrow\ZZ _{2\ZZ }[\X,\X[2]] \mathop{\rightarrow}\limits^{\pi} \ZZ _{2\ZZ }[\X,\X[2]]/I \rightarrow 0$.}\]
\end{example}

Another consequence of Theorem~\ref{schvaluation1} is the following result.

\begin{theorem}[Syzygy theorem for a coherent valuation ring with nonzero zerodivisors and a divisibility test]\label{hilbvaluation0divis}
    Let \(M=\Hm/U\) be a finitely presented  \(\V[\uX]\)-module, where \(\V\) is a coherent valuation ring with a divisibility test and nonzero zerodivisors. Assume that, w.r.t.\ {\color{OliveGreen}some} monomial order, $\LT(U)$ is finitely generated. Then \(M\) admits
a resolution by finite free \(\V[\uX]\)-modules
\[
  \cdots \stackrel{\varphi_{p+3}}\longrightarrow F_{p} \stackrel{\varphi_{p+2}}\longrightarrow F_{p} \stackrel{\varphi_{p+1}}\longrightarrow F_{p} \stackrel{\varphi_{p}}\longrightarrow F_{p-1} \stackrel{\varphi_{p-1}}\longrightarrow \cdots \stackrel{\varphi_{2}}\longrightarrow F_{1} \stackrel{\varphi_{1}}\longrightarrow F_{0} \stackrel{\varphi_{0}}\longrightarrow M \longrightarrow 0
\]
such that for some  \(p \leq n+1\),
\begin{itemize}
\item  \(\LT (\Ker(\varphi_{p}))=\bigoplus_{j=1}^{m_{p}}\gen{b_{j}}\epsilon_{j}\) with  \(b_{1},\dots,b_{m_{p}} \in \V\) and
\((\epsilon_{1},\dots,\epsilon_{m_{p}})\)  a basis for \(F_{p}\),
\item   \(\LT (\Ker(\varphi_{p+2k-1}))=\bigoplus_{j=1}^{m_{p}}\Ann(b_{j})\epsilon_{j}\) for $k\geq 1$,
\item  \(\LT (\Ker(\varphi_{p+2k}))=\bigoplus_{j=1}^{m_{p}}\Ann(\Ann(b_{j}))\epsilon_{j}\) for $k\geq 1$,
\end{itemize}
and at each step where indeterminates remain present, the considered monomial order is \foreignlanguage{german}{Schreyer}'s monomial order (as in the proof of Theorem~\ref{hilbvaluation1}).
\end{theorem}

\begin{proof} The part
\[
  F_{p} \stackrel{\varphi_{p}}\longrightarrow F_{p-1} \stackrel{\varphi_{p-1}}\longrightarrow \cdots \stackrel{\varphi_{2}}\longrightarrow F_{1} \stackrel{\varphi_{1}}\longrightarrow F_{0} \stackrel{\varphi_{0}}\longrightarrow M \longrightarrow 0
\]
of the free $\V[\uX]$-resolution with $p \leq n+1$ and $ \LT (\Ker(\varphi_{p}))=\bigoplus_{j=1}^{m_{p}}\gen{b_{j}}\epsilon_{j}$ follows from the proof of Theorem~\ref{hilbvaluation1}. W.l.o.g.,\ the $b_{j}$'s are $\neq0$. Let us denote by $\lst{g_{1},\dots,g_{m_{p}}}$ a Gröbner basis for $\Ker(\varphi_{p})$ such that $ \LT (g_{j})=b_{j}\epsilon_{j}$ for $1 \leq j\leq m_{p}$. So $\rS(g_{i},g_{j})=0$ for $i < j$. Thus the fact that $ \LT (\Ker(\varphi_{p+1}))=\bigoplus_{j=1}^{m_{p}}\Ann(b_{j})\epsilon_{j}$, $ \LT (\Ker(\varphi_{p+2}))=\bigoplus_{j=1}^{m_{p}}\Ann(\Ann(b_{j}))\epsilon_{j}$, etc.\ follows immediately from Theorem~\ref{schvaluation1}. Finally, let us recall
the equality $\Ann(\Ann(\Ann(I)))=\Ann(I)$ for an ideal~$I$.
\end{proof}

Let us point out that this shows that the free resolution is in general not a finite one.

\begin{corollary}[Syzygy theorem for a zero-dimensional coherent valuation ring with a divisibility test]\label{hilbvaluation0dim}
  Let \(M=\Hm/U\) be a finitely presented  \(\V[\uX]\)-module, where \(\V\) is a zero-dimensional coherent valuation ring\footnote{Note that a zero-dimensional ring without nonzero zerodivisors is a discrete field.} with a divisibility test. Then \(M\) admits
a free \(\V[\uX]\)-resolution as described in~Theorem~\ref{hilbvaluation0divis}.
\end{corollary}
\begin{example}\label{ex1} Let $g_{1}=\X[2]^{4}-\X[2],\,g_{2}=2\+\X[2],\,g_{3}=\X^{3}-1 \in (\ZZ /4\+\ZZ )[\X,\X[2]]$, and let us use the lexicographic order
$>_{1}$ for which $\X[2]>_{1}\X$. We have
\[
  \begin{aligned}
\rS(g_{1},g_{1})&=0\+ g_{1}=0\text,&\rS(g_{1},g_{2})&=2\+g_{1}-\X[2]^{3}g_{2}= -2\+\X[2]=-g_{2}\text,\\
\rS(g_{2},g_{2})&=2\+g_{2}=0\text,&\rS(g_{2},g_{3})&=\X^{3}g_{2}-2\+\X[2]g_{3}= 2\+\X[2]=g_{2}\text,\\
\rS(g_{3},g_{3})&=0\+ g_{3}=0\text,&\rS(g_{1},g_{3})&=\X^{3}g_{1}-\X[2]^{4}g_{3}= \X[2]^{4}-\X[2]\X^{3}=g_{1}-\X[2]g_{3}\text.
\end{aligned}
\]
Thus $(g_{1},g_{2},g_{3})$ is a (pseudo-reduced) Gröbner basis for $I=\gen{g_{1},g_{2},g_{3}}$ and $\LT (I)=\gen{\X[2]^{4},2\+\X[2],\X^{3}}$. By Theorem~\ref{schvaluation1},
$u_{1,3}=(\X^{3}-1,0,-\X[2]^{4}+\X[2])$, $u_{1,2}=(2,-\X[2]^{3}+1,0)$, $u_{2,3}=(0,\X^{3}-1,-2\+\X[2])$, $u_{2,2}=(0,2,0)$ form a (pseudo-reduced) Gröbner basis for the syzygy module $\Syz(g_{1},g_{2},g_{3})$ w.r.t.\ \foreignlanguage{german}{Schreyer}'s monomial order $>_{2}$ induced by~$>_{1}$ and  $\lst{g_{1},g_{2},g_{3}}$. In particular,
\[\makebox[\textwidth]{$
    \begin{aligned}
      \LT (\Syz(g_{1},g_{2},g_{3}))&=\gen{\LT (u_{1,3}),\dots, \LT (u_{2,2})}\\
      &= \gen{\X^{3}\epsilon_{1},2\+\epsilon_{1},\X^{3}\epsilon_{2},2\+\epsilon_{2}}= \gen{2,\X^{3}}\epsilon_{1} \oplus \gen{2,\X^{3}}\epsilon_{2}\text,
    \end{aligned}
$}\]
where $(\epsilon_{1},\epsilon_{2},\epsilon_{3})$ stands for the canonical basis of $(\ZZ /4\+\ZZ )[\X,\X[2]]^{3}$. We have
\[\makebox[\textwidth]{$
    \begin{aligned}
      \rS(u_{1,3},u_{1,3})&=0\+u_{1,3}=0\text,\\
      \rS(u_{1,3},u_{1,2})&=2\+u_{1,3}-\X^{3}u_{1,2}\\
      &=(-2,\X[2]^{3}\X^{3}-\X^{3},-2\+\X[2]^{4}+2\+\X[2])    \\
      &=-u_{1,2}+(\X[2]^{3}-1)u_{2,3}\text,&\\
      \rS(u_{1,3},u_{2,3})&=\rS(u_{1,3},u_{2,2})=0\text,\\
      \rS(u_{1,2},u_{1,2})&=2\+u_{1,2}=(0,-2\+\X[2]^{3}+2,0)\\
      &=(-\X[2]^{3}+1)u_{2,2}\text,
    \end{aligned}\quad
    \begin{aligned}
      \rS(u_{1,2},u_{2,3})&=\rS(u_{1,2},u_{2,2})=0\text,\\
      \rS(u_{2,3},u_{2,3})&=0\+u_{2,3}=0\text,\\
      \rS(u_{2,3},u_{2,2})&=2\+u_{2,3}-\X^{3}u_{2,2}\\
      &=(0,-2,0\+\X[2])\\
      &=(0,-2,0)\\
      &=-u_{2,2}\text,\\
      \rS(u_{2,2},u_{2,2})&=2\+u_{2,2}=0\text.
    \end{aligned}
$}\]
We recover that  $(u_{1,3},u_{1,2},u_{2,3},u_{2,2})$ is a Gröbner basis for $\Syz(g_{1},g_{2},g_{3})$. By Theorem~\ref{schvaluation1}, $u_{1,3;1,2}=(2,-\X^{3}+1,-\X[2]^{3}+1,0)$, $u_{1,2;1,2}=(0,2,0,\X[2]^{3}-1)$, $u_{2,3;2,2}=(0,0,2,-\X^{3}+1)$, $u_{2,2;2,2}=(0,0,0,2)$ form a (pseudo-reduced) Gröbner basis for the syzygy module $\Syz(u_{1,3},u_{1,2},u_{2,3},u_{2,2})$ w.r.t.\ \foreignlanguage{german}{Schreyer}'s monomial order $>_{3}$ induced by $>_{2}$ and  $\lst{u_{1,3},u_{1,2},u_{2,3},u_{2,2}}$. In particular,
\[\makebox[\textwidth]{$
    \begin{aligned}
      \LT (\Syz(u_{1,3},u_{1,2},u_{2,3},u_{2,2}))&=\gen{\LT(u_{1,3;1,2}),\dots,\LT(u_{2,2;2,2})}\\
      &= \gen{2\epsilon'_{1},\dots,2\+\epsilon'_{4}}= \gen{2}\epsilon'_{1} \oplus \gen{2}\epsilon'_{2} \oplus \gen{2}\epsilon'_{3} \oplus \gen{2}\epsilon'_{4}\text,
\end{aligned}
$}\]
where $(\epsilon'_{1},\dots,\epsilon'_{4})$ stands for the canonical basis of $(\ZZ /4\+\ZZ )[\X,\X[2]]^{4}$. By Theorem~\ref{schvaluation1}, we find four vectors $u_{(1,3;1,2),(1,3;1,2)},\dots,u_{(2,2;2,2),(2,2;2,2)}\in (\ZZ /4\+\ZZ )[\X,\X[2]]^{4}$ forming a (pseudo-reduced) Gröbner basis for the syzygy module $\Syz(u_{1,3;1,2},\dots,u_{2,2;2,2})$ w.r.t.\ \foreignlanguage{german}{Schreyer}'s monomial order $>_{4}$ induced by $>_{3}$ and  $\lst{u_{1,3;1,2},\dots,u_{2,2;2,2}}$. In particular,
\[\makebox[\textwidth]{$
    \begin{aligned}
      \LT (\Syz(u_{1,3;1,2},\dots,u_{2,2;2,2}))&=\gen{\LT (u_{(1,3;1,2),(1,3;1,2)}),\dots, \LT (u_{(2,2;2,2),(2,2;2,2)})}\\
      &= \gen{2}\epsilon'_{1} \oplus \gen{2}\epsilon'_{2} \oplus \gen{2}\epsilon'_{3} \oplus \gen{2}\epsilon'_{4}\text,
    \end{aligned}
$}\]
etc. We conclude that $I$ admits
the free $(\ZZ /4\+\ZZ )[\X,\X[2]]$-resolution
\[
  \cdots \stackrel{\varphi_{3}}\longrightarrow  (\ZZ /4\+\ZZ )[\X,\X[2]]^{4} \stackrel{\varphi_{2}}\longrightarrow (\ZZ /4\+\ZZ )[\X,\X[2]]^{4} \stackrel{\varphi_{1}}\longrightarrow (\ZZ /4\+\ZZ )[\X,\X[2]]^{3} \stackrel{\varphi_{0}}\longrightarrow I \longrightarrow 0
\]
such that   $ \LT (\Ker(\varphi_{i}))=\gen{2}\epsilon'_{1} \oplus \gen{2}\epsilon'_{2} \oplus \gen{2}\epsilon'_{3} \oplus \gen{2}\epsilon'_{4}$ for $i\geq 1$.

\end{example}

\section{The syzygy theorem and Schreyer's algorithm for a Bézout ring}


As explained in the proof of Theorem~\ref{Grob-Buch2}, one can avoid branching when computing a dynamical Gröbner basis (see~\cite{HY,Ye2,Y5}) for a Bézout domain of Krull dimension~$\leq 1$ (e.g.\ $\ZZ $ and the ring of all algebraic integers---note that the last one is not a PID) or a zero-dimensional coherent Bézout ring.
Note that this is not possible for  Prüfer domains of Krull dimension~$\leq 1$ which are not Bézout domains (e.g.\ $\ZZ [\sqrt{-5}]$, see~\cite[Section~4]{HY}).

Let us now generalise the results of Section~\ref{sec5} to the case of coherent strict Bézout rings.

\begin{theorem}[Schreyer's algorithm for Bézout rings]\label{schbez}
  We consider a coherent strict Bézout ring~$\R$ with a divisibility test.
  Let \(U\) be a submodule of \(\Hm\) with Gröbner basis \(\lst{g_{1},\dots,g_{p}}\).
  Then the relations~\(u_{i,j}\) computed by Algorithm~\ref{genSyzygyalg} form a Gröbner basis for the syzygy module \(\Syz(g_{1},\dots,
  g_{p})\) w.r.t.\ \foreignlanguage{german}{Schreyer}'s monomial order induced by~\(>\) and \(\lst{g_{1},\dots,g_{p}}\).
\end{theorem}

\begin{proof}
This follows directly from the local case given by Theorem~\ref{schvaluation1}: see the proof of Theorem~\ref{Grob-Buch2} for an explanation.
\end{proof}

\begin{theorem}[Syzygy theorem for a Bézout domain with a divisibility test]\label{hilbBezout1}\leavevmode
  Let \(M=\Hm/U\) be a finitely presented  \(\R[\uX]\)-module, where \(\R\) is a Bézout domain with a divisibility test.
Assume that, w.r.t.\ {\color{OliveGreen}some} monomial order, $\LT(U)$ is finitely generated. Then \(M\) admits
a finite free \(\R[\uX]\)-resolution
\[
  0 \rightarrow F_{p} \rightarrow F_{p-1} \rightarrow \cdots \rightarrow F_{1} \rightarrow F_{0} \rightarrow M \rightarrow 0
\]
of length \(p \leq n+1\).
\end{theorem}
\begin{proof}
This follows directly from the local case.
\end{proof}

\begin{corollary}[Syzygy theorem for a one-dimensional  Bézout domain with a divisibility test]\label{hilbbezou0}\leavevmode
Let \(M=\Hm/U\) be a finitely presented  \(\R[\uX]\)-module, where \(\R\) is a Bézout domain of Krull dimension~$\leq 1$ with a divisibility test. Then \(M\) admits
a finite free \(\R[\uX]\)-resolution
\[
  0 \rightarrow F_{p} \rightarrow F_{p-1} \rightarrow \cdots \rightarrow F_{1} \rightarrow F_{0} \rightarrow M \rightarrow 0
\]
of length \(p \leq n+1\).
\end{corollary}

Let us now treat the case of zero-dimensional coherent Bézout rings.

\begin{theorem}[Syzygy theorem for a zero-dimensional  Bézout ring with a divisibility test]\label{hilbbezou0znz}
  Let \(M=\Hm/U\) be a finitely presented  \(\R[\uX]\)-module, where $\R$ is a coherent zero-dimensional Bézout ring with a divisibility test. Then \(M\) admits
  a free \(\R[\uX]\)-resolution
  \[
    \cdots \stackrel{\varphi_{p+3}}\longrightarrow F_{p} \stackrel{\varphi_{p+2}}\longrightarrow F_{p} \stackrel{\varphi_{p+1}}\longrightarrow F_{p} \stackrel{\varphi_{p}}\longrightarrow F_{p-1} \stackrel{\varphi_{p-1}}\longrightarrow \cdots \stackrel{\varphi_{2}}\longrightarrow F_{1} \stackrel{\varphi_{1}}\longrightarrow F_{0} \stackrel{\varphi_{0}}\longrightarrow M \longrightarrow 0
  \]
  such that for some  \(p \leq n+1\),
\begin{itemize}
\item  \(\LT (\Ker(\varphi_{p}))=\bigoplus_{j=1}^{m_{p}}\gen{b_{j}}\epsilon_{j}\) with \(b_{1},\dots,b_{m_{p}} \in \R \) and \((\epsilon_{1},\dots,\epsilon_{m_{p}})\) a basis for \(F_{p}\),
\item   \(\LT (\Ker(\varphi_{p+2k-1}))=\bigoplus_{j=1}^{m_{p}}\Ann(b_{j})\epsilon_{j}\) for $k\geq 1$,
\item  \(\LT (\Ker(\varphi_{p+2k}))=\bigoplus_{j=1}^{m_{p}}\Ann(\Ann(b_{j}))\epsilon_{j}\) for $k\geq 1$,
\end{itemize}
and at each step where indeterminates remain present, the considered monomial order is \foreignlanguage{german}{Schreyer}'s monomial order.
\end{theorem}

\begin{proof}
  This follows directly from the local case.
\end{proof}

\subsection{The case of
  the integers}


The following theorems are particular cases  of Theorem~\ref{schbez} and Corollary~\ref{hilbbezou0} for $\R=\ZZ$.

\begin{theorem}[Schreyer's algorithm for $\R=\ZZ$]\label{schint}
  Let \(U\) be a submodule of \(\Hm\) with Gröbner basis \(\lst{g_{1},\dots,g_{p}}\).
  Then the relations  \(u_{i,j}\) computed by Algorithm~\ref{genSyzygyalg} form a Gröbner basis for the syzygy module \(\Syz(g_{1},\dots,
  g_{p})\) w.r.t.\ \foreignlanguage{german}{Schreyer}'s monomial order induced by~\(>\) and \(\lst{g_{1},\dots,g_{p}}\). Moreover,  for \(1 \leq  i <  j \leq p\) such that  \(\LP(g_{i})=\LP(g_{j})\), we have
  \[
    \LT (u_{i,j})=\tfrac{\LC(g_{j})}{\gcd(\LC(g_{i}),\LC(g_{j}))}\uX^{(\mdeg(g_{j})-\mdeg(g_{i}))^{+}}\epsilon_{i}\text.
  \]
\end{theorem}

\begin{theorem}[Syzygy theorem for $\R=\ZZ$]\label{hilbint1}
Let \(M\) be a finitely generated  \(\ZZ [\uX]\)-module. Then \(M\) admits
a finite free \(\ZZ [\uX]\)-resolution
\[
  0 \rightarrow F_{p} \rightarrow F_{p-1} \rightarrow \cdots \rightarrow F_{1} \rightarrow F_{0} \rightarrow M \rightarrow 0
\]
of length \(p \leq n+1\).

\end{theorem}

\begin{example}\label{ex0002}
Let $g_{1}=\X[2]^{2}-\X+3,\,g_{2}=4\X^{2}-4,\,g_{3}=6\X+6 \in \ZZ [\X,\X[2]]$, and let us use the lexicographic order
$>_{1}$ for which $\X[2]>_{1}\X$. We have:
\[
  \begin{aligned}
  \rS(g_{1},g_{2})&=4\X^{2}g_{1}-\X[2]^{2}g_{2}=4\+g_{1}+(-\X+3)g_{2}\text,\\
  \rS(g_{1},g_{3})&=6\X g_{1}-\X[2]^{2}g_{3}= -6\+g_{1}+(-\X+3)g_{3}\text,\\
  \rS(g_{2},g_{3})&=3\+g_{2}-2\X g_{3}= -2\+g_{3}\text.
\end{aligned}
\]
Thus $(g_{1},g_{2},g_{3})$ is a Gröbner basis for $I=\gen{g_{1},g_{2},g_{3}}$ and $\LT (I)=\gen{\X[2]^{2},4\X^{2},6\X}$. By Theorem~\ref{schint},
$u_{1,2}=(4\X^{2}-4,-\X[2]^{2}+\X-3,0)$, $u_{1,3}=(6\X+6,0,-\X[2]^{2}+\X-3)$, $u_{2,3}=(0,3,-2\X+2)$ form a  Gröbner basis for the syzygy module $\Syz(g_{1},g_{2},g_{3})$ w.r.t.\ \foreignlanguage{german}{Schreyer}'s monomial order $>_{2}$ induced by~$>_{1}$ and  $\lst{g_{1},g_{2},g_{3}}$. In particular,
\[\begin{split}
    \LT (\Syz(g_{1},g_{2},g_{3}))=\gen{\LT (u_{1,2}),\LT (u_{1,3}), \LT (u_{2,3})}= \gen{4\X^{2}\epsilon_{1},6\X \epsilon_{1},3\+\epsilon_{2}}\\
    \qquad= \gen{4\X^{2},6\X}\epsilon_{1} \oplus \gen{3}\epsilon_{2}
    =2\gen{2\X^{2},3\X}\epsilon_{1} \oplus \gen{3}\epsilon_{2}=2\gen{\X^{2},3\X}\epsilon_{1} \oplus \gen{3}\epsilon_{2}\text,
  \end{split}
  \]
where $(\epsilon_{1},\epsilon_{2},\epsilon_{3})$ stands for the canonical basis of $\ZZ [\X,\X[2]]^{3}$. Thus
\begin{equation*}
\makebox[\textwidth]{$u'_{1,2}= {\X u_{1,3}-u_{1,2}}=({2\X^{2}+6\X+4},\allowbreak {\X[2]^{2}-\X+3},\allowbreak {-\X[2]^{2}\X+\X^{2}-3\X}), u_{1,3}, u_{2,3}$}
\end{equation*}
form
a reduced Gröbner basis for $\Syz(g_{1},g_{2},g_{3})$. We have:
\[
  \begin{aligned}
    \rS(u'_{1,2},u_{1,3})&=3\+u'_{1,2}-\X u_{1,3}=2\+u_{1,3}+(\X[2]^{2}-\X+3)u_{2,3}\text,\\
    \rS(u'_{1,2},u_{2,3})&=\rS(u_{1,3},u_{2,3})=0\text.
\end{aligned}\]
We recover that  $(u'_{1,2},u_{1,3},u_{2,3} )$ is a Gröbner basis for $\Syz(g_{1},g_{2},g_{3})$. By Theorem~\ref{schint}, $u_{1,2;1,3}=(3,-\X-2,-\X[2]^{2}+\X-3)$ forms a (pseudo-reduced) Gröbner basis for the syzygy module $\Syz(u'_{1,2},u_{1,3},u_{2,3})$ w.r.t.\ \foreignlanguage{german}{Schreyer}'s monomial order~$>_{3}$ induced by $>_{2}$ and  $\lst{u'_{1,2},u_{1,3},u_{2,3}}$. In particular, $\LT (\Syz(u'_{1,2},u_{1,3},u_{2,3}))=\gen{\LT (u_{1,2;1,3})}= \gen{3}\epsilon'_{1} $
where $(\epsilon'_{1},\epsilon'_{2},\epsilon'_{3})$ stands for the canonical basis of $\ZZ [\X,\X[2]]^{3}$. It follows that $\Syz(u'_{1,2},u_{1,3},u_{2,3})$ is free. We conclude that  $I$ admits
the following length-$2$ free $\ZZ [\X,\X[2]]$-resolution:
\[\makebox[\textwidth]{$
    0 \rightarrow \ZZ[\X,\X[2]] \xrightarrow{u_{1,2;1,3}} \ZZ[\X,\X[2]]^{3} \xrightarrow{
      \left(\begin{smallmatrix}
          {\color{OliveGreen}u'_{1,2}}\\
          u_{1,3}\\
          u_{2,3}
        \end{smallmatrix}\right)} \ZZ[\X,\X[2]]^{3} \xrightarrow{
      \left(\begin{smallmatrix}
          g_{1}\\
          g_{2}\\
          g_{3}
        \end{smallmatrix}\right)} I \rightarrow 0$.}
\]
\end{example}

\subsection{The case of \texorpdfstring{$\ZZ /N\ZZ$}{Z/NZ}}\label{s4}

The elements of $\ZZ/N\ZZ$ are simply written as integers (their representatives in $ \intervalles $). When talking about the gcd of two nonzero elements in $\ZZ/N\ZZ$ we mean the gcd of their representatives in $ \intervalless $. For a nonzero element $a$ in $\ZZ/N\ZZ$, letting $b= \gcd(N, a)$, the class of $\dfrac{N}{b}$ in $\ZZ/N\ZZ$ will be denoted by $\ann(a)$; it generates $\Ann(a)$.

\begin{itemize}
\item The Division algorithm~\ref{gendivalg} attains its goal: the gcd and the Bézout identity to be found in line~\ref{gendivalg:5} will be computed by finding $d,b,b_{i}$ $(i\in D)$ in $\ZZ$ such that $d=\gcd(N,\gcd\sotq{\LC(h_{i})}{i\in D})=bN+\sumtq{b_{i}\LC(h_{i})}{i\in D}$; the euclidean division in line~\ref{gendivalg:5} will be performed in~$\ZZ$;
\item The S-polynomial algorithm~\ref{genSalg} attains its goal: note that in this case, the generator of the annihilator of~$\LC(f)$ to be found on line~\ref{genSalg:2} may be taken to be~$\ann(\LC(f))$, so that the auto-S-polynomial of~$f$~is
\[
  \rS(f,f)=\ann(\LC(f))f\text;
\]
\item Buchberger's algorithm~\ref{genBuchbergeralg} attains its goal.
\end{itemize}

The following theorems are particular cases of Theorems~\ref{schbez} and~\ref{hilbbezou0znz} for $\R=\ZZ /N\ZZ $.

\begin{theorem}[Schreyer's algorithm for $\R=\ZZ /N\ZZ $]\label{schintn}
  Let
  \(U\) be a submodule of \(\Hm\) with Gröbner basis
  \(\lst{g_{1},\dots,g_{p}}\).  Then the relations \(u_{i,j}\) computed
  by Algorithm~\ref{genSyzygyalg} form a Gröbner basis for the
  sy\-zy\-gy mo\-dule \(\Syz(g_{1},\dots, g_{p})\) w.r.t.\
  \foreignlanguage{german}{Schreyer}'s monomial order induced
  by~\(>\) and \(\lst{g_{1},\dots,g_{p}}\). Moreover, for all
  \(1 \leq i\leq j \leq p\) such that \(\LP(g_{i})=\LP(g_{j})\), we have
  \[
    \LT (u_{i,j})=
    \begin{cases}
      \ann(\LC(g_{i}))\epsilon_{i}&\text{if $i=j$,}\\[.3em]
      \tfrac{\LC(g_{j})}{\gcd(\LC(g_{i}),\LC(g_{j}))}\uX^{(\mdeg(g_{j})-\mdeg(g_{i}))^{+}}\epsilon_{i}&\text{otherwise.}
    \end{cases}
  \]
\end{theorem}

\begin{theorem}[Syzygy theorem for $\R=\ZZ /N\ZZ$]\label{hilbvaluation0znz}
    Let \(M\) be a finitely presented  \((\ZZ /N\ZZ )[\uX]\)-module. Then \(M\) admits
  a free \((\ZZ /N\ZZ )[\uX]\)-resolution
  \[
    \cdots \stackrel{\varphi_{p+3}}\longrightarrow F_{p} \stackrel{\varphi_{p+2}}\longrightarrow F_{p} \stackrel{\varphi_{p+1}}\longrightarrow F_{p} \stackrel{\varphi_{p}}\longrightarrow F_{p-1} \stackrel{\varphi_{p-1}}\longrightarrow \cdots \stackrel{\varphi_{2}}\longrightarrow F_{1} \stackrel{\varphi_{1}}\longrightarrow F_{0} \stackrel{\varphi_{0}}\longrightarrow M \longrightarrow 0
  \]
  such that for some  \(p \leq n+1\),
  \[
    \begin{aligned}
      \LT (\Ker(\varphi_{p}))&=\bigoplus_{j=1}^{m_{p}}\gen{b_{j}}\epsilon_{j}\text,&
      \LT (\Ker(\varphi_{p+1}))&=\bigoplus_{j=1}^{m_{p}}\frac{N}{\gcd(N, b_{j})}\epsilon_{j}\text,\\
      \LT (\Ker(\varphi_{p+2}))&=\bigoplus_{j=1}^{m_{p}}\gen{b_{j}}\epsilon_{j}\text,&
      \LT (\Ker(\varphi_{p+3}))&=\bigoplus_{j=1}^{m_{p}}\frac{N}{\gcd(N, b_{j})}\epsilon_{j}\text{, etc.,}
    \end{aligned}
  \]
  where \((\epsilon_{1},\dots,\epsilon_{m_{p}})\) is a basis for \(F_{p}\), \(b_{1},\dots,b_{m_{p}} \in \ZZ /N\ZZ \), and the considered monomial order is \foreignlanguage{german}{Schreyer}'s monomial order.
\end{theorem}

\begin{example}\label{ex000216}
  Let $g_{1}=\X[2]+1$, $g_{2}=\X^{3}+\X^{2}+6$, $g_{3}=3\X^{2}$, $g_{4}=9$  in $(\ZZ /12\+\ZZ)[\X,\X[2]]$, and let us use the lexicographic order
  $>_{1}$ for which $\X[2]>_{1}\X$. We have
  \[\makebox[\textwidth]{$
      \begin{aligned}
        \rS(g_{1},g_{1})&=0\+g_{1}=0\text,\\
        \rS(g_{1},g_{2})&=\X^{3}g_{1}-\X[2]g_{2}=(-\X^{2}-6)g_{1}+g_{2}\text,\\
        \rS(g_{1},g_{3})&=3\X^{2}g_{1}-\X[2]g_{3}= g_{3}\text,\\
        \rS(g_{1},g_{4})&=9\+g_{1}-\X[2]g_{4}= g_{4}\text,\\
        \rS(g_{2},g_{2})&=0\+g_{2}=0\text,
      \end{aligned}\,
      \begin{aligned}
        \rS(g_{2},g_{3})&=3\+g_{2}-\X g_{3}= g_{3}+2\+g_{4}\text,\\
        \rS(g_{2},g_{4})&=9\+g_{2}-\X^{3}g_{3}= (\X^{2}+6)g_{4}\text,\\
        \rS(g_{3},g_{3})&=4\+g_{3}=0\text,\\
        \rS(g_{3},g_{4})&=3\+g_{3}-\X^{2}g_{4}= 0,\\
        \rS(g_{4},g_{4})&=4\+g_{4}=0\text.
      \end{aligned}
  $}\]
  Thus $(g_{1},g_{2},g_{3},g_{4})$ is a (pseudo-reduced) Gröbner basis for $I=\gen{g_{1},g_{2},g_{3},g_{4}}$ and $\LT (I)=\gen{\X[2],\X^{3},3\X^{2},9}$. By Theorem~\ref{schintn},
  $u_{1,2}=(\X^{3}+\X^{2}+6,-\X[2]-1,\allowbreak0,0)$, $u_{1,3}=(3\X^{2},0,-\X[2]-1,0)$, $u_{1,4}=(9,0,0,-\X[2]-1)$, $u_{2,3}=(0,3,-\X-1,\allowbreak-2)$, $u_{2,4}=(0,9,-\X^{3},-\X^{2}-6)$, $u_{3,3}=(0,0,4,0)$, $u_{3,4}=(0,0,3,-\X^{2})$, $u_{4,4}=(0,0,0,4)$ form a  Gröbner basis for the syzygy module $\Syz(g_{1},g_{2},g_{3},g_{4})$ w.r.t.\ \foreignlanguage{german}{Schreyer}'s monomial order~$>_{2}$ induced by $>_{1}$ and  $\lst{g_{1},g_{2},g_{3},g_{4}}$. In particular,
  \[\makebox[\textwidth]{$
      \begin{aligned}
        \LT (\Syz(g_{1},g_{2},g_{3},g_{4}))&=\gen{\LT (u_{1,2}),\dots, \LT (u_{4,4})}\\
        &= \gen{\X^{3},3\X^{2},9}\epsilon_{1} \oplus \gen{3,9}\epsilon_{2}\oplus \gen{4,3}\epsilon_{3}\oplus \gen{4}\epsilon_{4}\\
        &=\gen{\X^{3},3}\epsilon_{1} \oplus \gen{3}\epsilon_{2}\oplus \gen{1}\epsilon_{3}\oplus \gen{4}\epsilon_{4}\text,
      \end{aligned}
  $}\]
  where $(\epsilon_{1},\epsilon_{2},\epsilon_{3},\epsilon_{4})$ stands for the canonical basis of $(\ZZ /12\+\ZZ)[\X,\X[2]]^{4}$. Thus $u_{1,2}$, $u'_{1,4}=-u_{1,4}=(3,0,0,\X[2]+1)$, $u_{2,3}$, $u'_{3,3}=u_{3,3}-u_{3,4}=(0, 0, 1,\X^{2})$, $u_{4,4}$ form
  a reduced Gröbner basis for $\Syz(g_{1},g_{2},g_{3},g_{4})$. We have
  \[\makebox[\textwidth]{$
      \begin{aligned}
        \rS(u_{1,2},u'_{1,4})&=3\+u_{1,2}-\X^{3}u'_{1,4}\\
        &=
        \begin{aligned}[t]
          (\X^{2}+2)u'_{1,4}+(3\+\X[2]+3)u_{2,3}
          +(3\+\X[2]\X+3\+\X[2]+3\X+3)u'_{3,3}\\
          +(2\+\X[2]\X^{3}+2\+\X[2]\X^{2}+2\X^{3}+2\X^{2}+\X[2]+1)u_{4,4}\rlap,
        \end{aligned}\\
        \rS(u'_{1,4},u'_{1,4})&=4\+u'_{1,4}=(\X[2]+1)u_{4,4}\text,\\
        \rS(u_{2,3},u_{2,3})&=4\+u_{2,3}=(8\X+8)u'_{3,3}+(\X^{3}+\X^{2}+1)u_{4,4}\text,\\
        \rS(u_{4,4},u_{4,4})&=3\+u_{4,4}=0\text.
      \end{aligned}
      $}\] By Theorem~\ref{schintn}, the elements
  $u_{1,2;1,4}=(3,-\X^{3}-\X^{2}-2,{-3\X[2]-3},\allowbreak-3\X[2]\X-3\X[2]-3\+\X-3,-2\X[2]\X^{3}-2\X[2]\X^{2}-\X[2]-2\+\X^{3}-2\+\X^{2}-1)$,
  $u_{1,4;1,4}=(0,4,0,0,-\X[2]-1)$,
  $u_{2,3;2,3}=(0,0,4,\allowbreak{-8\X-8},\allowbreak-\X^{3}-\X^{2}-1)$,
  $u_{4,4;4,4}=(0,0,0,0,3)$ form a (pseudo-reduced) Gröbner basis for
  the syzygy module $\Syz(u_{1,2},u'_{1,4},u_{2,3},u'_{3,3},u_{4,4})$
  w.r.t.\ \foreignlanguage{german}{Schreyer}'s monomial order $>_{3}$
  induced by $>_{2}$ and
  $\lst{u_{1,2},u'_{1,4},u_{2,3},u'_{3,3},u_{4,4}}$. In particular,
  $\LT (\Syz(u_{1,2},u'_{1,4},u_{2,3},u'_{3,3},u_{4,4}))=
  \gen{3}\epsilon'_{1} \oplus \gen{4}\epsilon'_{2} \oplus \gen{4}\epsilon'_{3}
  \oplus \gen{3}\epsilon'_{5}$, where $(\epsilon'_{1},\dots,\epsilon'_{5})$
  stands for the canonical basis of $(\ZZ /12\+\ZZ)[\X,\X[2]]^{5}$.

  We conclude that $I$ admits the free $\R[\X,\X[2]]$-resolution ($\R=\ZZ /12\+\ZZ$)
  \[\makebox[\textwidth]{$
      \cdots \stackrel{\varphi_{4}}\longrightarrow  \R[\X,\X[2]]^{4} \stackrel{\varphi_{3}}\longrightarrow  \R[\X,\X[2]]^{4} \stackrel{\varphi_{2}}\longrightarrow \R[\X,\X[2]]^{5} \stackrel{\varphi_{1}}\longrightarrow \R[\X,\X[2]]^{4} \stackrel{\varphi_{0}}\longrightarrow I \rightarrow 0
  $}\]
  with   $ \LT (\Ker(\varphi_{2i}))=\gen{4}\epsilon''_{1} \oplus \gen{3}\epsilon''_{2} \oplus \gen{3}\epsilon''_{3} \oplus \gen{4}\epsilon''_{4}$ and
  $ \LT (\Ker(\varphi_{2i+1}))=\gen{3}\epsilon''_{1} \oplus \gen{4}\epsilon''_{2} \oplus \gen{4}\epsilon''_{3} \oplus \gen{3}\epsilon''_{4}$ for $i\geq 1$, where $(\epsilon''_{1},\dots,\epsilon''_{4})$ stands for the canonical basis of~$\R[\X,\X[2]]^{4}$.
\end{example}
\addcontentsline{toc}{section}{References}
\providecommand\MR[1]{}


\begin{thebibliography}{10}
\providecommand{\url}[1]{\texttt{#1}}
\providecommand{\urlprefix}{URL }
\expandafter\ifx\csname urlstyle\endcsname\relax
  \providecommand{\doi}[1]{doi:\discretionary{}{}{}#1}\else
  \providecommand{\doi}{doi:\discretionary{}{}{}\begingroup
  \urlstyle{rm}\Url}\fi
\providecommand{\selectlanguage}[1]{\relax}

\def\bbland{and}
\def\bbletal{et~al.}
\def\bbleditors{editors}        \def\bbleds{eds.}
\def\bbleditor{editor}          \def\bbled{ed.}
\def\bbledby{edited by}
\def\bbledition{edition}        \def\bbledn{ed.}
\def\bblvolume{volume}          \def\bblvol{vol.}
\def\bblof{of}
\def\bblnumber{number}          \def\bblno{no.}
\def\bblin{in}
\def\bblpages{pages}            \def\bblpp{pp.}
\def\bblpage{page}              \def\bblp{p.}
\def\bblchapter{chapter}        \def\bblchap{chap.}
\def\bbltechreport{Technical Report}
\def\bbltechrep{Tech. Rep.}
\def\bblmthesis{Master's thesis}
\def\bblphdthesis{Ph.D. thesis}
\def\bblfirst{first}            \def\bblfirsto{1st}
\def\bblsecond{second}          \def\bblsecondo{2nd}
\def\bblthird{third}            \def\bblthirdo{3rd}
\def\bblfourth{fourth}          \def\bblfourtho{4th}
\def\bblfifth{fifth}            \def\bblfiftho{5th}
\def\bblst{st}  \def\bblnd{nd}  \def\bblrd{rd}
\def\bblth{th}
\def\bbljan{January}  \def\bblfeb{February}  \def\bblmar{March}
\def\bblapr{April}    \def\bblmay{May}       \def\bbljun{June}
\def\bbljul{July}     \def\bblaug{August}    \def\bblsep{September}
\def\bbloct{October}  \def\bblnov{November}  \def\bbldec{December}
\def\bblins{in}
\def\bblinj{in}

\newcommand{\Capitalize}[1]{\uppercase{#1}}
\newcommand{\capitalize}[1]{\expandafter\Capitalize#1}

\bibitem{AL}
William~W. Adams \bbland{} Philippe Loustaunau.
\newblock \emph{An introduction to {G}r\"{o}bner bases}, \emph{Graduate Studies
  in Mathematics}, \bblvol{}~3.
\newblock American Mathematical Society, Providence, 1994.
\newblock \doi{10.1090/gsm/003}.

\bibitem{Bi67}
Errett Bishop.
\newblock \emph{Foundations of constructive analysis}.
\newblock McGraw-Hill, New York, 1967.

\bibitem{BB85}
Errett Bishop \bbland{} Douglas Bridges.
\newblock \emph{Constructive analysis},
  \emph{\foreignlanguage{german}{Grundlehren der mathematischen
  Wissenschaften}}, \bblvol{} 279.
\newblock Springer, Berlin, 1985.
\newblock \doi{10.1007/978-3-642-61667-9}.

\bibitem{Buchbe0}
Bruno Buchberger.
\newblock {\selectlanguage{german}\emph{\foreignlanguage{german}{Ein
  {A}lgorithmus zum {A}uffinden der {B}asiselemente des {R}estklassenringes
  nach einem nulldimensionalen {P}olynomideal}}}.
\newblock Ph.{D}.\ thesis, \foreignlanguage{german}{Mathematisches Institut,
  Universit{\"a}t Innsbruck}, 1965.
\newblock Translation by Michael P. Abramson: An algorithm for finding the
  basis elements of the residue class ring of a zero dimensional polynomial
  ideal, \emph{J. Symbolic Comput.}, 41(3--4)\string:475--511, 2006.
  \MR{2202562}.

\bibitem{coxlittleoshea05}
David~A. Cox, John Little, \bbland{} Donal O'Shea.
\newblock \emph{Using algebraic geometry}, \emph{Graduate Texts in
  Mathematics}, \bblvol{} 185.
\newblock Springer, New York, second \bbledn{}, 2005.
\newblock \doi{10.1007/b138611}.

\bibitem{CLO}
David~A. Cox, John Little, \bbland{} Donal O'Shea.
\newblock \emph{Ideals, varieties, and algorithms: an introduction to
  computational algebraic geometry and commutative algebra}.
\newblock Undergraduate Texts in Mathematics, Springer, Cham, fourth \bbledn{},
  2015.
\newblock \doi{10.1007/978-3-319-16721-3}.

\bibitem{DLQ2014}
Gema-Maria D\'{\i}az-Toca, Henri Lombardi, \bbland{} Claude Quitt\'e.
\newblock \emph{Modules sur les anneaux commutatifs: cours et exercices}.
\newblock Calvage \& Mounet, Paris, 2014.

\bibitem{EH}
Viviana Ene \bbland{} J{\"u}rgen Herzog.
\newblock \emph{Gr{\"o}bner bases in commutative algebra}, \emph{Graduate
  Studies in Mathematics}, \bblvol{} 130.
\newblock American Mathematical Society, Providence, 2012.

\bibitem{GY}
Maroua Gamanda \bbland{} Ihsen Yengui.
\newblock Noether normalization theorem and dynamical {G}r\"{o}bner bases over
  {B}ezout domains of {K}rull dimension 1.
\newblock \emph{J. Algebra}, 492,  52--56, 2017.
\newblock \doi{10.1016/j.jalgebra.2017.09.002}.

\bibitem{HY}
Amina Hadj~Kacem \bbland{} Ihsen Yengui.
\newblock Dynamical {G}r\"obner bases over {D}edekind rings.
\newblock \emph{J. Algebra}, 324(1),  12--24, 2010.
\newblock \doi{10.1016/j.jalgebra.2010.04.014}.

\bibitem{ACMC}
Henri Lombardi \bbland{} Claude Quitt{\'e}.
\newblock \emph{Commutative algebra: constructive methods. Finite projective
  modules}, \emph{Algebra and Applications}, \bblvol{}~20.
\newblock Springer, Dordrecht, 2015.
\newblock \doi{10.1007/978-94-017-9944-7}.
\newblock Translated from the French (Calvage \& Mounet, Paris, 2011, revised
  and extended by the authors) by Tania K. Roblot.

\bibitem{MRR}
Ray Mines, Fred Richman, \bbland{} Wim Ruitenburg.
\newblock \emph{A course in constructive algebra}.
\newblock Universitext, Springer, New York, 1988.
\newblock \doi{10.1007/978-1-4419-8640-5}.

\bibitem{MoY}
Samiha Monceur \bbland{} Ihsen Yengui.
\newblock On the leading terms ideals of polynomial ideals over a valuation
  ring.
\newblock \emph{J. Algebra}, 351,  382--389, 2012.
\newblock \doi{10.1016/j.jalgebra.2011.11.015}.

\bibitem{redei35}
L{\'a}szl{\'o} R\'edei.
\newblock Ein kombinatorischer {S}atz.
\newblock \emph{Acta Sci. Math. (Szeged)}, 7,  39--43, 1934--1935.
\newblock \urlprefix\url{http://acta.bibl.u-szeged.hu/13432}.

\bibitem{Sc}
Frank-Olaf Schreyer.
\newblock {\selectlanguage{german}\emph{\foreignlanguage{german}{Die
  {B}erechnung von {S}yzygien mit dem verallgemeinerten {W}eierstra{\ss}schen
  {D}ivisionssatz und eine {A}nwendung auf analytische {C}ohen-{M}acaulay
  {S}tellenalgebren minimaler {M}ultiplizit{\"a}t}}}.
\newblock \bblmthesis{}, \foreignlanguage{german}{Universit{\"a}t Hamburg},
  1980.

\bibitem{Ye2}
Ihsen Yengui.
\newblock Dynamical {G}r\"{o}bner bases.
\newblock \emph{J. Algebra}, 301(2),  447--458, 2006.
\newblock \doi{10.1016/j.jalgebra.2006.01.051}.

\bibitem{Y2}
Ihsen Yengui.
\newblock Corrigendum to ``{D}ynamical {G}r\"{o}bner bases'' and to
  ``{D}ynamical {G}r\"{o}bner bases over {D}edekind rings''.
\newblock \emph{J. Algebra}, 339(1),  370--375, 2011.
\newblock \doi{10.1016/j.jalgebra.2011.05.004}.

\bibitem{Y4}
Ihsen Yengui.
\newblock The {G}r\"obner ring conjecture in the lexicographic order case.
\newblock \emph{Math. Z.}, 276(1--2),  261--265, 2014.
\newblock \doi{10.1007/s00209-013-1197-y}.

\bibitem{Y5}
Ihsen Yengui.
\newblock \emph{Constructive commutative algebra: projective modules over
  polynomial rings and dynamical {G}r{\"o}bner bases}, \emph{Lecture Notes in
  Mathematics}, \bblvol{} 2138.
\newblock Springer, Cham, 2015.
\newblock \doi{10.1007/978-3-319-19494-3}.

\end{thebibliography}

\end{document}